\documentclass[11pt]{article}

\usepackage{fullpage}
\usepackage[utf8]{inputenc}
\usepackage[T1]{fontenc}  
\usepackage{todonotes}
\usepackage{amsmath}
\usepackage{amsfonts,amssymb}
\usepackage[margin=0.72in]{geometry}
\usepackage{url,hyperref}
\usepackage{amsthm}
\usepackage{xcolor}
\usepackage{float}
\usepackage{booktabs} 
\usepackage{graphicx}
\usepackage{caption}
\usepackage{subcaption}
\usepackage{ulem}
\usepackage{bm}
\usepackage{algorithm,algorithmic}
\allowdisplaybreaks

\newcommand{\BEAS}{\begin{eqnarray*}}
\newcommand{\EEAS}{\end{eqnarray*}}
\newcommand{\BEQ}{\begin{equation}}
\newcommand{\EEQ}{\end{equation}}
\newcommand{\BIT}{\begin{itemize}}
\newcommand{\EIT}{\end{itemize}}

\newcommand{\eg}{{\it e.g.}}
\newcommand{\ie}{{\it i.e.}}
\newcommand{\cf}{{\it cf. }}

\def\<#1,#2>{\langle #1,#2\rangle}

\newtheorem{theorem}{Theorem}
\newtheorem{remark}{Remark}
\newtheorem{assumption}{Assumption}
\newtheorem{lemma}[theorem]{Lemma}
\newtheorem{corollary}[theorem]{Corollary}
\newtheorem{proposition}[theorem]{Proposition}
\theoremstyle{definition}
\newtheorem{definition}{Definition}[section]

{\begin{quote}}{\end{quote}}

\makeatletter
\long\def\@makecaption#1#2{
   \vskip 9pt
   \begin{small}
   \setbox\@tempboxa\hbox{{\bf #1:} #2}
   \ifdim \wd\@tempboxa > 5.5in
        \begin{center}
        \begin{minipage}[t]{5.5in}
        \addtolength{\baselineskip}{-0.95pt}
        {\bf #1:} #2 \par
        \addtolength{\baselineskip}{0.95pt}
        \end{minipage}
        \end{center}
   \else
	\hbox to\hsize{\hfil\box\@tempboxa\hfil}
   \fi
   \end{small}\par
}
\makeatother

\newcounter{oursection}

\newcounter{lecture}

\newcommand{\norm}[1]{{\left\vert\kern-0.25ex\left\vert\kern-0.25ex\left\vert #1
    \right\vert\kern-0.25ex\right\vert\kern-0.25ex\right\vert}}

\author{
Anran Hu 
\thanks{Department of IEOR, Columbia University. 
\textbf{Email:} \texttt{ah4277@columbia.edu}}
\thanks{This work was carried out primarily while the author was at the Mathematical Institute, University of Oxford, supported by a Hooke Fellowship.}
\and
Junzi Zhang
\thanks{Citadel Securities. 
\textbf{Email:} \texttt{junzizmath@gmail.com}} 
}
\title{MF-OML: Online Mean-Field Reinforcement Learning with Occupation Measures for Large Population Games}
\date{}

\providecommand{\keywords}[1]
{
  \small	
  \textbf{Key words.} #1
}

\begin{document}

\maketitle

\begin{abstract}
Reinforcement learning for multi-agent games has attracted lots of attention recently. However, given the challenge of solving Nash equilibria, existing works with guaranteed polynomial complexities either focus on variants of zero-sum and potential games, or aim at solving (coarse) correlated equilibria, or require access to simulators, or rely on certain assumptions that are hard to verify. This work proposes \texttt{MF-OML} (Mean-Field Occupation-Measure Learning), an online mean-field reinforcement learning  algorithm for computing approximate Nash equilibria of large population sequential symmetric games.  \texttt{MF-OML} is the first fully polynomial multi-agent reinforcement learning algorithm for provably solving Nash equilibria (up to mean-field approximation gaps that vanish as the number of players $N$ goes to infinity) beyond variants of zero-sum and potential games.  When evaluated by the cumulative deviation from Nash equilibria, the algorithm is shown to achieve a high probability regret bound of $\tilde{O}(M^{3/4}+N^{-1/2}M)$ for games with the strong Lasry-Lions monotonicity condition, and a regret bound of $\tilde{O}(M^{11/12}+N^{-1/6}M)$ for games with only the Lasry-Lions monotonicity condition, where $M$ is the total number of episodes and $N$ is the number of agents of the game. As a by-product, we also obtain the first tractable globally convergent computational algorithm for computing approximate Nash equilibria of monotone mean-field games.
\end{abstract}

{\keywords{mean-field games, symmetric $N$-player games, Nash equilibrium, occupation measure, multi-agent reinforcement learning, operator splitting}}


\section{Introduction}
In the domain of game theory, multi-agent systems present both profound opportunities and formidable challenges. These systems, where multiple agents interact under a set of strategic decision-making rules, are pivotal in fields ranging from economics \cite{liu2022welfare} to autonomous vehicle navigation \cite{zhang2020multi}. The inherent complexity and dynamic nature of these interactions pose significant challenges, primarily due to the scale of agent populations and the complexity of each agent's strategy space \cite{yang2020overview,zhang2021multi}.

Introduced by the seminal work of  \cite{huang2006large} and  \cite{lasry2007mean}, mean-field games (MFGs) provide an ingenious way of finding the approximate Nash equilibrium solution to the otherwise notoriously hard 
 $N$-player stochastic games. It
offers a powerful framework to tackle the complexities of multi-agent systems. MFGs simplify the analysis of large populations by considering the limit where the number of agents approaches infinity, allowing for the modeling of each agent's interaction with the average effect of the rest of the population instead of individual agents. This limiting regime has been shown to provide both analytical and computational tools for finding approximate Nash equilibria for symmetric $N$-player games, especially when the number of agents $N$ is large. It has been established that the Nash equilibria of MFGs can be used to construct $O(1/\sqrt{N})$-Nash equilibria for the corresponding $N$-player game \cite{huang2006large,lasry2007mean,delarue2019master,saldi2018markov}.
This approach has been extensively explored in the literature, with applications noted in areas such as financial markets, crowd dynamics, and large-scale social networks. 

In recent years, there is a surge of interest in introducing reinforcement learning (RL) to solving MFGs, which then serves as a tool for multi-agent reinforcement learning (MARL) \cite{yang18,subramanian2019reinforcement,guo2019learning}.   

\paragraph{Reinforcement Learning.} 
Reinforcement Learning is a segment of machine learning where an agent learns to make decisions by performing actions and receiving feedback in the form of rewards or penalties. This learning paradigm enables agents to learn optimal policies through trial-and-error interactions with a dynamic environment. RL has been successfully applied to various problems, including complex games like Go and practical applications such as robotics and sequential decision-making tasks. RL can be categorized into three settings: simulator setting, online setting and offline setting. Offline RL involves learning from a fixed dataset without the ability to gather new data \cite{kumar2020conservative,yu2020mopo,nachum2019algaedice}. In the simulator setting, one is allowed to have access to some simulator that can provide samples of next state and reward given any current state, any current action at any time \cite{gheshlaghi2013minimax,sidford2023variance}. Online RL, on the other hand, does not have such access. One needs to interact with the environment in real-time to collect samples by adopting certain strategies \cite{auer2008near,azar2017minimax,jin2018q,agarwal2021theory}. Online RL emphasizes the need for algorithms that can efficiently balance exploration (\ie, trying new actions to discover potentially better strategies) and exploitation (\ie, leveraging known strategies to maximize rewards), a balance that is critical in dynamic and uncertain environments. Compared to the simulator setting, online RL usually faces larger challenges in the design and analysis of the algorithms due to the fact that one can not modify the underlying system/environment and observe samples that are beneficial to the learning procedure. However, in many applications, it is hard to build reliable simulators, where designing online RL algorithms becomes crucial.


\paragraph{MARL and mean-field RL.}
The integration of RL with multi-agent settings has garnered significant interest in the recent decade.  However, given the difficulties of finding Nash equilibrium in the general multi-agent systems, extending algorithms and analyses to the RL setting imposes further challenges. The theoretical works on MARL either focus on variants of zero-sum games \cite{bai2020near,liu2021sharp,kalogiannis2024zero,anagnostides2024optimistic} and potential games \cite{leonardos2021global,ding2022independent,guo2023markov}, or aim at solving (coarse) correlated equilibria \cite{jin2021v,song2021can,muller2021learning,muller2022learning} or subjective equilibria \cite{yongacoglu2022independent}. For general multi-agent systems, RL algorithms for finding Nash equilibrium usually suffer from exponential computational complexities \cite{hu2003nash,lanctot2017unified,cui2022provably}. Another stream of works on MARL deals with RL in MFGs to potentially tackle the curse of many agents in the traditional MARL literature. Most of these works propose and analyze RL algorithms that require access to mean-field simulators, {\color{black}under either contractivity assumptions \cite{guo2019learning,anahtarci2019fitted,cui2021learning,xie2021learning,guo2023general,fu2019actor,zaman2023oracle}, monotonicity assumptions \cite{perrin2020fictitious,perolat21,lee2021reinforcement,fabian2023learning,zhang2023learning}, or access to a finite model class which contains the true model \cite{huang2024statistical,huang2024model}. 
For the online setting (with $N$-player environments), \cite{yang18,subramanian2020multi,subramanian2020partially} propose Nash Q-learning based mean-field RL algorithms that require some stringent oracle assumptions inherited from \cite{hu2003nash}. More recently,   
\cite{yardim2022policy} studies independent learning in regularized contractive $N$-player mean-field games but suffers from an inherent gap in addition to the mean-field approximation error due to the heavy regularization needed for the contractive convergence analysis framework; \cite{yardim2023stateless} removes such inherent gaps by focusing on monotone $N$-player mean-field games but is restricted to stage games. 
} 

\paragraph{Our work and contributions.}
In this work, we aim to solve online reinforcement learning for a class of large population sequential symmetric games that adhere to the classical Lasry-Lions monotonicity condition. In Section \ref{sec:mf-omi}, we introduce the \texttt{MF-OMI-FBS} (Mean-Field-Occupation Measure Inclusion with Forward-Backward Splitting) algorithm. This algorithm is designed to find Nash equilibrium of the limiting MFG by transforming the Nash equilibrium search into a monotone inclusion problem over occupation measures. We establish convergence results for this algorithm when the MFG model is fully known.
Then, in Section \ref{sec:mf-oml}, we propose \texttt{MF-OML} (Mean-Field-Occupation Measure Learning) which is extended from the \texttt{MF-OMI-FBS} algorithm to address $N$-player models where the models are unknown. To address the exploration-exploitation trade-off in the online RL setting, in each episode, we randomly select one agent to follow a fully exploratory policy, while the remaining agents implement policies derived from the normalized updates of an approximated version of \texttt{MF-OMI-FBS}. We are able to show that when evaluated by the cumulative deviation from Nash equilibria, the algorithm achieves a high probability regret bound of $\tilde{O}(M^{3/4}+N^{-1/2}M)$ for games with the strong Lasry-Lions monotonicity condition, where $M$ is the total number of episodes and $N$ is the number of agents of the game. This bound includes two components: the first, $\tilde{O}(M^{3/4})$, arises from the learning procedure itself, while the second, $\tilde{O}(N^{-1/2}M)$, results from the mean-field approximation error. For games that only satisfy the basic Lasry-Lions monotonicity condition, our algorithm exhibits a regret bound of $\tilde{O}(M^{11/12}+N^{-1/6}M)$.

{\color{black}
 To sum up, the contributions of this paper are two-fold:
 \begin{itemize}
     \item We propose \texttt{MF-OMI-FBS}, the first tractable globally convergent computational algorithm with fully polynomial iteration complexities for solving the Nash equlibrium of (monotone) MFGs without certain uniqueness assumptions or introducing (non-vanishing) regularization. 
     Moreover, \texttt{MF-OMI-FBS} is simple and efficient to implement and naturally lends itself to generic MFGs that are beyond those studied under our theoretical framework. Empirical benchmarks against existing algorithms demonstrate that \texttt{MF-OMI-FBS} consistently outperforms the baselines.

     \item We propose \texttt{MF-OML}, the first fully polynomial online MARL algorithm for provably solving Nash equilibrium, up to mean-field approximation gaps that vanish as the number of players $N$ goes to infinity, beyond variants of zero-sum and potential games. We conduct numerical experiments to verify the sub-linear growth of regret and performance improvement as $N$ grows. 
 \end{itemize}
 }


\paragraph{Technical challenges and novelties.}  
Our method diverges significantly from most existing literature on RL for MFGs, which typically updates directly on policies. By converting the problem of finding a Nash equilibrium into one of identifying the corresponding occupation measure, our approach facilitates the use of optimization tools from the broader literature. Moreover, it is crucial to effectively transfer from the space of occupation measures back into the policy space when designing online RL algorithms. This dual transformation is key to our approach and underscores its novelty. 

In addition, the presence of dynamics introduces inherent new difficulties in online RL compared to stage games \cite{yardim2023stateless} as it introduces an unknown constraint set over occupation measures that needs to be estimated/learned. It is well-known that optimization problems are sensitive to the perturbations of constraint sets, making the analysis of learning algorithms difficult. To resolve this issue, we identify a nearly unconstrained reformulation of the projection onto the set of the occupation measures (\cf Lemma \ref{projection_d_vs_L}) and obtain a robust optimization problem which facilitates the analysis.


\paragraph{Other related works.} The idea of transforming the problem of finding Nash equilibria to an optimization problem over occupation measures follows from \cite{guo2024mf}, which 
obtains an optimization framework for solving Nash equilibria for MFGs and establishes local convergence results for the projected gradient descent algorithm. Our work focuses on a special class of monotone MFGs and designs algorithms based on operator splitting that achieve global convergence, which is further extended to the online RL setting. The idea of utilizing occupation measure has also been adopted in continuous time MFGs to obtain new existence results for relaxed
Nash equilibria and computational benefits \cite{bouveret2020mean,dumitrescu2023linear}. It is later used to study the sensitivity analysis of Stackelberg MFGs \cite{guo2022optimization}. The method of monotone operator splitting is also used in \cite{liu2020splitting,briceno2019implementation,nurbekyan2024monotone} to design computational algorithms for solving Nash equilibria in continuous-time monotone MFGs with fully known models.

{\color{black}
The computation of Nash equilibrium for (discrete-time) MFGs have been widely studied in the literature. However, existing works in the literature either establish convergence results only for the continuous-time limits of the proposed algorithms \cite{perrin2020fictitious,perolat21}, 
    require certain uniqueness assumptions \cite{guo2019learning,angiuli21,angiuli2022unified,subramanian2022decentralized,angiuli2023convergence} and (non-vanishing) regularization  \cite{cui2021approximately,xie2021learning,zhang2023learning}\footnote{Note that \cite{cui2021approximately,xie2021learning} suffers from inherent gaps due to the need for heavy regularization to ensure contractivity. In contrast, the analysis of \cite{zhang2023learning} is based on monotonicity and hence the regularization effect is less inherent. Nevertheless, the distance metric adopted in \cite{zhang2023learning} is a pseudo-metric given that the (unique regularized) Nash equilibrium mean-field distribution is not necessarily all positive, and hence cannot really guarantee convergence of the computed policies in general.}, 
    make stringent oracle assumptions that the Nash equilibrium of the stage mean-field games in all iterations can be solved exactly \cite{muller2021learning}, or focus on potential games \cite{geist2021concave} or the much easier (coarse) correlated equilibrium instead of Nash equilibrium \cite{muller2022learning}. In contrast, \texttt{MF-OMI-FBS} does not suffer from these limitations and neatly tackles all these restrictions.
}

{\color{black}\paragraph{Outline.} The paper begins by introducing the setting of $N$-player games and mean-field games (MFGs) in Section \ref{setup_section}. In Section \ref{sec:mf-omi}, we present the mean-field occupation measure inclusion framework for solving MFGs and establish the convergence properties of Algorithm \ref{MF-OMI-FBS-Cons} in Theorem \ref{FBS_convergence}. Section \ref{sec:mf-oml} extends this framework to an online learning setting for $N$-player games, culminating in Algorithm \ref{MF-OML} and the regret analysis provided in Theorem \ref{thm:mf_oml_regret}. 
}



\section{$N$-player games and mean-field games}\label{setup_section}


\paragraph{Problem setup.} We consider symmetric Markov $N$-player games, with a common finite state space $\mathcal{S}=\{1,\dots,S\}$, a common finite action space $\mathcal{A}=\{1,\dots,A\}$, a finite planning horizon $T\geq 1$ and an initial state profile ${\bf s}_0=(s_0^1,\dots,s_0^N)\in\mathcal{S}^N$. Being symmetric, the rewards and dynamics of any agent  in the game depend only on the state-action pair of the agent, and the set of state-action pairs of the population (regardless of the order), or equivalently the empirical state-action distribution of the population. More precisely,  
at time $t\in\mathcal{T}:=\{0,\dots,T-1\}$, each agent $i\in[N]:=\{1,\dots,N\}$ receives a random reward $r_t(s_t^i,a_t^i,L_t^N)$ (with its expectation denoted as $R_t(s_t^i,a_t^i,L_t^N)$) when the game is at state profile ${\bf s}_t$ and taking action profile ${\bf a}_t$, and transitions to state $s_{t+1}^i$ independently (over the $N$ agents) with transition probability $P_t(s_{t+1}^i|s_t^i,a_t^i)$, where ${\bf s}_t=(s_t^1,\dots,s_t^N)\in\mathcal{S}^N$, ${\bf a}_t=(a_t^1,\dots,a_t^N)\in\mathcal{A}^N$, and $L_t^N(s,a)=\frac{1}{N}\sum_{i\in[N]}{\bf 1}\{s_t^i=s,a_t^i=a\}$ ($s\in\mathcal{S},\,a\in\mathcal{A}$) is the empirical distribution of the population. 

For any agent $i\in[N]$, an admissible strategy/policy $\pi_t$ ($t\in\mathcal{T}$) is a mapping from $\mathcal{S}$ to $\Delta(\mathcal{A})$, where $\Delta(\mathcal{X})$ is the set of probability distributions over $\mathcal{X}$. 
We denote the set of all such strategy/policy sequences $\pi=\{\pi_t\}_{t\in\mathcal{T}}$ as $\Pi$. When agent $i$ takes policy $\pi^i\in\Pi$ ($i\in[N]$), at each time $t\in\mathcal{T}$, given the agent states $s_t^i$ ($i\in[N]$), the actions $a_t^i\sim\pi_t^i(s_t^i)$ ($i\in[N]$) of all agents are taken independently. Note that here we consider the case where all agents take randomized/relaxed local policies depending only on their own local state, which is suitable for decentralized execution that is efficient in the large-population settings we focus on. For notational flexibility, we use $\pi_t(a_t|s_t)$ and $\pi_t(s_t,a_t)$ exchangeably to denote the $a_t$-th dimension of the probability vector $\pi_t(s_t)$. The goal is to find a Nash equilibrium (NE) of the game, which is defined below.
\begin{definition}[Nash equilibrium (NE) of a $N$-player game]
    A strategy profile $\boldsymbol{\pi}=\{\pi^i\}_{i\in[N]}$ with $\pi^i=\{\pi_t^i\}_{t\in\mathcal{T}}\in\Pi$ is called a Nash equilibrium (NE) of the $N$-player game if and only if $\texttt{NashConv}(\boldsymbol{\pi})=0$. Here \texttt{NashConv} is defined as
\begin{equation}\label{nashconv_def}
\texttt{NashConv}(\boldsymbol{\pi}):=\dfrac{1}{N}\sum\nolimits_{i\in[N]}\left(\max_{\tilde{\pi}^i\in\Pi}V^i(\pi^1,\dots,\tilde{\pi}^i,\dots,\pi^N)-V^i(\boldsymbol{\pi})\right),
\end{equation}
where $V^i(\boldsymbol{\pi}):=\mathbb{E}_{\boldsymbol{\pi}}\left[\sum_{t\in\mathcal{T}}r_t(s_t^i,a_t^i,L_t^N)\right]$ is the expected cumulative reward of agent $i$ (with initial state $s_0^i$), and 
the expectation is over the trajectory of states and actions when the agents take independent actions  $a_t^j\sim \pi_t(s_t^j)$ for $j\in[N]$ ($t\in\mathcal{T}$).
\end{definition}

At this equilibrium point, no player has the incentive to unilaterally deviate from their chosen strategy, given the strategies chosen by the others. Intuitively, $\texttt{NashConv}$ characterizes the aggregated single-agent side sub-optimality of a strategy profile, and an NE is a strategy profile with no such sub-optimality. 
It is well-known that showing the existence of NEs and finding them  are difficult for general $N$-player games \cite{saldi2018markov}. To alleviate this, a relaxation of the Nash equilibrium concept is introduced.
\begin{definition}[$\epsilon$-Nash equilibrium of an $N$-player game]
   A strategy profile $\boldsymbol{\pi}$ is called a $\epsilon$-Nash equilibrium (NE) if $\texttt{NashConv}(\boldsymbol{\pi})\leq \epsilon$. 
\end{definition}

For notational simplicity, when the strategy profile $\boldsymbol{\pi}=\{\pi^i\}_{i\in[N]}$ is symmetric, namely $\pi^i=\pi\in\Pi$  for all $i\in[N]$, we also denote $\texttt{NashConv}(\pi):=\texttt{NashConv}(\boldsymbol{\pi})$. We say that $\pi$ is an ($\epsilon$-)NE of the $N$-player game if and only if the strategy profile $\boldsymbol{\pi}:=\{\pi^i\}_{i\in[N]}$ with $\pi^i=\pi$ is an ($\epsilon$-)NE.

\paragraph{Mean-field games.} As an approximate and limiting model of the aforementioned symmetric Markov $N$-player game with a large population size $N$, we introduce a mean-field game (MFG) with the same finite time horizon $T<\infty$, finite state space $\mathcal{S}$ and finite action space $\mathcal{A}$. Such a game consists of an infinite number of symmetric/anonymous players, and a representative player independently starts with an initial state $s_0\sim \mu_0^N\in\Delta(\mathcal{S})$, where $\mu_0^N(s)=\frac{1}{N}\sum_{i\in[N]}{\bf 1}\{s_0^i=s\}$ ($s\in\mathcal{S}$) is the empirical initial state distribution of the aforementioned $N$-player game being approximated. For any policy sequence $\pi\in\Pi$, we denote by $L^\pi=\{L_t^{\pi}\}_{t\in\mathcal{T}}$ as the mean-field flow induced from $\pi$, defined recursively as
\begin{equation}\label{Gamma_pi}
\begin{split}
    L_0^\pi(s,a)&=\mu_0^N(s)\pi_0(a|s),\\
    L_{t+1}^{\pi}(s',a')&=\pi_{t+1}(a'|s')\sum_{s\in\mathcal{S},a\in\mathcal{A}}P_t(s'|s,a)L_t^\pi(s,a),\quad t\in\{0,\dots,T-2\}.
\end{split}
\end{equation}
Namely, $L_t^\pi$ denotes the joint state-action distribution among all players at time $t$, and can be viewed as an approximation of the empirical state-action distribution $L_t^N$ of the population under policy sequence $\pi$. In each time step $t\in\mathcal{T}$, the representative player takes an action $a_t\in\mathcal{A}$ following the randomized policy $\pi_t(\cdot|s_t)\in\Delta(\mathcal{A})$, receives a reward $r_t(s_t,a_t,L_t^\pi)$ and moves to a new state $s_{t+1}$ following the transition probabiltiy $P_t(\cdot|s_t,a_t)$. 
Similar to the $N$-player games being approximated, we can define the (mean-field) Nash equilibrium (NE) solution concept of the MFG here. A policy sequence $\pi=\{\pi_t\}_{t\in\mathcal{T}}$ is an NE of the MFG if and only if its exploitability $\texttt{Expl}(\pi)=0$, where the exploitabilty is defined as 
\begin{equation}\label{exploitablity}
\texttt{Expl}(\pi)=\max_{\pi'\in\Pi}V^{\pi'}(L^\pi)-V^{\pi}(L^\pi).
\end{equation}
Here given any mean-field flow $L=\{L_t\}_{t\in\mathcal{T}}\subseteq\Delta(\mathcal{S}\times\mathcal{A})$, $V^{\pi}(L)$ is defined as the expected total reward of $\pi$ of the $L$-induced Markov decision process (MDP) $\mathcal{M}(L)$ with rewards $\tilde{r}_t^L(s_t,a_t):=r_t(s_t,a_t,L_t)$ and transitions $P_t(s_{t+1}|s_t,a_t)$. 
Similarly, a policy sequence $\pi$ is called an $\epsilon$-NE if and only if $\texttt{Expl}(\pi)\leq \epsilon$. Intuitively, the exploitability of a policy sequence characterizes the room for unilateral improvement of any representative player, mimicking the counterpart definition for $N$-player games. Note that the exploitability is always non-negative by definition.

\paragraph{Outstanding assumptions and approximation guarantees.} Throughout this paper, we consider the following assumptions of Lipschitz continuity  and monotonicity (on the expected rewards). 
\begin{assumption}\label{outstanding_assumptions_1}
The expected rewards 
are $C_R$-Lipschitz continuous in $L_t\in\Delta(\mathcal{S}\times\mathcal{A})$ for any $s\in\mathcal{S},a\in\mathcal{A},t\in\mathcal{T}$, namely $|R_t(s,a,L_t^{(1)})-R_t(s,a,L_t^{(2)})|\leq C_R\|L_t^{(1)}-L_t^{(2)}\|_1$ for any $L_t^{(1)},L_t^{(2)}\in\Delta(\mathcal{S}\times\mathcal{A})$. 
Furthermore, we also assume that the random rewards are a.s. bounded, \ie, $|r_t(s,a,L)|\leq R_{\max}<\infty$ a.s. for all $s\in\mathcal{S},a\in\mathcal{A},t\in\mathcal{T},L\in\Delta(\mathcal{S}\times\mathcal{A})$.


\end{assumption}
\begin{assumption}\label{outstanding_assumptions_2}
The expected rewards are $\lambda$-Lasry-Lions monotone\footnote{We call it Lasry-Lions monotone as it was first introduced in the seminal work of \cite{lasry2007mean} on mean-field games. This is also to distinguish it from another commonly adopted monotone assumption in the differentiable normal-form/stage games literature \cite{lin2020finite}, which assumes monotonicity of the reward gradients w.r.t. actions.} ($\lambda\geq 0$) in the sense that for any $\{L_t^{(1)}\}_{t\in\mathcal{T}}$, $\{L_t^{(2)}\}_{t\in\mathcal{T}}\subseteq\Delta(\mathcal{S}\times\mathcal{A})$, we have 
\[
\sum_{t\in\mathcal{T},s\in\mathcal{S},a\in\mathcal{A}}(R_t(s,a,L_t^{(1)})-R_t(s,a,L_t^{(2)}))(L_t^{(1)}(s,a)-L_t^{(2)}(s,a))\leq -\lambda\sum_{t\in\mathcal{T},s\in\mathcal{S},a\in\mathcal{A}}(L_t^{(1)}(s,a)-L_t^{(2)}(s,a))^2.
\]
\end{assumption}
Note that when $\lambda>0$, the property in Assumption \ref{outstanding_assumptions_2} is typically referred to as strongly (Lasry-Lions) monotone. When $\lambda=0$, Assumption \ref{outstanding_assumptions_2} is a slight generalization of the standard (Lasry-Lions) monotonicity assumption in the MFG literature. More precisely, except for being implicitly adopted in \cite{guo2024mf}, the literature on monotone MFGs are largely restricted to rewards that depend on $\mu_t:=\sum_{a\in\mathcal{A}}L_t(\cdot,a)$ \cite{elie2020convergence,perrin2020fictitious,perolat21,geist2021concave,zhang2023learning}. Furthermore, throughout the paper, whenever Assumption \ref{outstanding_assumptions_2} is assumed, we can indeed replace it with a weaker assumption that only require that the expected rewards are monotone on ``reachable'' (or induced) mean-field flows in the sense that for any two policies $\pi^1,\pi^2\in\Pi$, we have 
\[
\sum_{t\in\mathcal{T},s\in\mathcal{S},a\in\mathcal{A}}(R_t(s,a,L_t^{\pi^1})-R_t(s,a,L_t^{\pi^2}))(L_t^{\pi^1}(s,a)-L_t^{\pi^2}(s,a))\leq 0.
\]
But for clarity, we stick to slightly stronger Assumption \ref{outstanding_assumptions_2} which has a more consistent form compared to the monotonicity assumptions in the literature.

Under Assumption \ref{outstanding_assumptions_1}, it can be shown that an NE solution exists for the MFG \cite{saldi2018markov,cui2021approximately,guo2024mf}. In addition, we also have the following approximation guarantees of the approximating MFG for the original $N$-player game. The proof can be found in Section \ref{n2mfg_proof}.

\begin{theorem}\label{n2mfg}
Suppose that Assumption \ref{outstanding_assumptions_1} holds. Then if $\pi\in\Pi$ is an $\epsilon$-NE of the MFG, it is also an $\epsilon'$-NE of the original $N$-player game, where $\epsilon'=\epsilon+2C_R\sqrt{\dfrac{\pi}{2}}SAT/\sqrt{N} + C_RSAT/N$. 
\end{theorem}

\begin{remark}
The approximation guarantees of MFGs have been widely studied in the literature \cite{saldi2018markov,cui2021approximately,yardim2023stateless,yardim2022policy}. We provide the result and the proof here for self-containedness. We remark that here we do not need the assumption that players in the $N$-player game have i.i.d. initial state distributions which is commonly assumed in the literature. In addition, by utilizing the structure that the dynamics of the players are decoupled from each other (\cf \cite{perrin2020fictitious,perolat21,geist2021concave} for the same assumption), the approximation error we obtain depends linearly on the time horizon $T$, which is in sharp contrast to the exponential growth in the general setting \cite{yardim2024mean}. 

\end{remark}

\section{Solving mean-field Nash equilibria via occupation-measure inclusion}\label{sec:mf-omi}


In this section, we first study computational algorithms for finding the NEs of the MFG that is used to approximate the NEs of the original symmetric Markov $N$-player game, assuming full knowledge of the MFG. This serves as the stepping stone for designing the mean-field RL algorithm for solving the original $N$-player game in the episodic online RL setup in the next section. Throughout this section, to better conform to the MFG literature, we slightly relax the model assumption to allow the initial state to follow an arbitrary (fixed) distribution $\mu_0\in\Delta(\mathcal{S})$, with the empirical distribution $\mu_0^N$ as the special case. 

\subsection{Representing mean-field NE with occupation measure} 
Motivated by the appearance of the induced MDP in the definition of mean-field NE in \eqref{exploitablity}, we first recall the classical result that an MDP can be represented as a linear program of the occupation measure. Such an observation is first noted in \cite{manne1960linear} in the context of (single-agent) MDPs, and more recently applied to MFGs in \cite{guo2024mf} to propose an optimization-framework called MF-OMO (Mean-Field Occupation-Measure Optimization) for computing mean-field NEs. However, due to the inherent non-convexity of the quadratic penalty formulation, convergence of algorithms under the MF-OMO framework to  NE solutions is only shown when the initialization is sufficiently close. In this work, we instead propose a monotone inclusion framework of occupation measures, called MF-OMI (Mean-Field Occupation-Measure Inclusion) to exploit the underlying monotonicity of the MFG, which eventually leads to globally convergent algorithms to NEs. 



\paragraph{Occupation measure.} We begin by introducing a new variable $d_t(s,a)$ for any $t\in\mathcal{T}$, $s\in \mathcal{S},a\in\mathcal{A}$, which represents the occupation measure of the representative agent under some policy sequence $\pi=\{\pi_t\}_{t\in\mathcal{T}}$ in an  MDP with transition kernel $P_t(s_{t+1}|s_t,a_t)$ and initial state distribution $\mu_0$, \ie, $d_t(s,a)=\mathbb{P}(s_t=s,a_t=a)$, with $s_0\sim\mu_0$, $s_{t+1}\sim P_t(\cdot|s_t,a_t)$, $a_t\sim \pi_t(\cdot|s_t)$ for $t=0,\dots,T-2$. By definition, we have $d_t(s,a)=L_t^{\pi}(s,a)$. 
Given the occupation measure $d=\{d_t\}_{t\in\mathcal{T}}$, define a set-valued mapping $\texttt{Normalize}$ that retrieves the policy from the occupation measure. This mapping $\texttt{Normalize}$ maps from a sequence $\{d_t\}_{t\in\mathcal{T}}\subseteq\mathbb{R}_{\geq 0}^{SA}$ to a set of policy sequences 
    $\{\pi_t\}_{t\in\mathcal{T}}$:  
    for any $\{d_t\}_{t\in\mathcal{T}}\subseteq\mathbb{R}_{\geq 0}^{SA}$, $\pi\in\texttt{Normalize}(d)$ if and only if
       $ \pi_t(a|s)=\frac{d_t(s,a)}{\sum_{a'\in\mathcal{A}}d_t(s,a')}$
    when $\sum_{a'\in\mathcal{A}}d_t(s,a')>0$, and $\pi_t(\cdot|s)$ is an arbitrary probability distribution over $\mathcal{A}$ when  $\sum_{a'\in\mathcal{A}}d_t(s,a')=0$.

\paragraph{Compact notation.} To facilitate the presentation below, hereafter we define vectors $c_R(L)\in\mathbb{R}^{SAT}$ and $b\in\mathbb{R}^{ST}$ as 
    \begin{equation}\label{kkt_param1}
    c_R(L)=\left[
\begin{array}{c}
-R_0(\cdot,\cdot,L_0)\\
\vdots\\
-R_{T-1}(\cdot,\cdot,L_{T-1})
\end{array}
\right],\quad
b=\left[
\begin{array}{c}
0\\
\vdots\\
0\\
\mu_0
\end{array}
\right],
    \end{equation}
    with the subscript $R=\{R_t\}_{t\in\mathcal{T}}$ denoting the dependency on the expected rewards, and $R_t(\cdot,\cdot,L_t)\in\mathbb{R}^{SA}$ being a flattened vector (with column-major order) of the expected rewards $R_t(s,a,L_t)$.  
    In addition, we define the matrix $A_P\in\mathbb{R}^{ST\times SAT}$  as 
\begin{equation}\label{kkt_param2}
A_P=\left[
\begin{array}{ccccccc}
   W_0  & -Z &  0 & 0 & \cdots &0 & 0\\
   0 & W_1  & -Z &  0 & \cdots &0 & 0\\ 
   0 & 0 & W_2 & -Z & \cdots & 0 &0 \\
   \vdots & \vdots & \vdots & \vdots & \ddots & \vdots &\vdots \\
   0 & 0 & 0 & 0 & \cdots & W_{T-1} & -Z\\
   Z & 0 & 0 & 0 & \cdots & 0 & 0
\end{array}\right].
\end{equation}
Here the subscript $P=\{P_t\}_{t\in\mathcal{T}}$ denotes the dependency on the transition probabilities,  
$W_t\in \mathbb{R}^{S\times SA}$ is the matrix with the $s$-th row ($s=1,\dots,S$) being the  flattened vector  $[P_t(s|\cdot,\cdot)]\in\mathbb{R}^{SA}$ (with column-major order),  
and the matrix $Z$ is defined as 
\begin{equation}\label{eq:defn-Z}
 Z:=[\overbrace{I_{S},\dots,I_{S}]}^{A}\in \mathbb{R}^{S\times SA},  
\end{equation}
where $I_S$ is the identity matrix with dimension $S$. 
In addition, we also represent the mapping from $\pi$ to $L^\pi$ in \eqref{Gamma_pi} as $\Gamma(\pi;P)$ to make the dependency on $P$ explicit. 

For notational brevity, we alternatively use $c(L)$, $A$ and $L^{\pi}$ to denote $c_R(L)$, $A_P$ and $\Gamma(\pi;P)$ under the ground-truth rewards $R$ and transitions $P$, and mainly highlight the dependencies on $R$ and $P$ when they are replaced with their approximations. 
With slight abuse of notation, we also use  $d\in\mathbb{R}^{SAT}$ to denote the  flattened/vectorized in the same order as $c(L)$ from the original vector sequence $\{d_t\}_{t\in\mathcal{T}}\subseteq\mathbb{R}^{SA}$. Accordingly, hereafter both $d=\{d_t\}_{t\in\mathcal{T}}$ and $L=\{L_t\}_{t\in\mathcal{T}}$ are viewed as flattened vectors (with column-major order) in $\mathbb{R}^{SAT}$ or a sequence of $T$ flattened vectors (with column-major order) in $\mathbb{R}^{SA}$, depending on the context, and we use $L_t(s,a)$ and $L_{s,a,t}$ (resp. $d_t(s,a)$ and $d_{s,a,t}$) alternatively.

In \cite{guo2024mf}, it is shown that $\pi^\star\in\Pi$ is an NE of the MFG if and only if $\exists d^\star,L^\star\in\mathbb{R}^{SAT}$, such that $\pi^\star\in\texttt{Normalize}(d^\star)$ and that the following two conditions hold: (A) $d^\star$ solves the linear program which minimizes $c(L^\star)^\top d$ subject to $Ad=b$, $d\geq 0$; (B) $L^\star=d^\star$. Note that condition (A) is exactly the well-known occupation-measure linear program reformulation of the $L^\star$-induced MDP, as introduced at the beginning of this section.  

Last but not least, we adopt the following notation. For any vector $x\in\mathbb{R}^n$, we use $x_{i:j}\in\mathbb{R}^{j-i+1}$ to denote the sub-vector containing its $i$-th to $j$-th elements. For any closed convex set $C\subseteq\mathbb{R}^n$, its normal cone operator $\mathcal{N}_C(\cdot)$ is 
defined as a set-valued mapping with $\mathcal{N}_C(u)=\emptyset$ for $u\notin C$ and $\mathcal{N}_C(u)=\{v\in\mathbb{R}^n|v^\top(w-u)\leq 0,\,\forall w\in C\}$ for $u\in C$. In addition, the dual cone of $C$ is defined as the set $C^*=\{y\in\mathbb{R}^n|y^\top x\geq 0,\,\forall x\in C\}$.


\paragraph{MF-OMI: Mean-field occupation-measure inclusion.} We now introduce MF-OMI, an inclusion problem formulation of MFGs with occupation-measure variables, which we will utilize to exploit the underlying monotonicity of the MFGs to obtain globally convergent algorithms to NE solutions. 
\begin{theorem}[MF-OMI]\label{thm:mf-omi}
Finding an NE solution to the MFG is equivalent to solving the following inclusion problem, referred to as MF-OMI (Mean-Field Occupation-Measure Inclusion):
\begin{equation}\label{mf-omi}
  \text{Find $d$ such that }  0\in c(d)+\mathcal{N}_{\{x|Ax=b,x\geq 0\}}(d), \tag{MF-OMI}
\end{equation}
More precisely, if $d$ is a solution to \eqref{mf-omi}, then any $\pi\in\texttt{Normalize}(d)$ is an NE of the original MFG. And if $\pi$ is an NE of the original MFG, then $L^\pi$ is a solution to \eqref{mf-omi}. 
\end{theorem}

We now show that when the original MFG exhibits monotonicity properties, then MF-OMI is a monotone inclusion, namely both $c$ and the normal cone operator $\mathcal{N}_{\{x|Ax=b,x\geq 0\}}$ are monotone. 
Before we proceed, let us first recall the following definitions of monotonicity properties from the generic operator theory \cite{ryu2016primer}.  
\begin{definition}[Monotone operator]
An operator/mapping $G:\mathbb{R}^q\rightarrow\mathbb{R}^q$ ($q\in\mathbb{N}$) is said to be monotone on $\mathcal{X}\subseteq \mathbb{R}^q$ if for any two $x_1,\,x_2\in\mathcal{X}$, $(G(x_1)-G(x_2))^\top(x_1-x_2)\geq 0$. 
It is said to be $\rho$-strongly monotone on $\mathcal{X}$ if for any two $x_1,\,x_2\in\mathcal{X}$, $(G(x_1)-G(x_2))^\top(x_1-x_2)\geq \rho \|x_1-x_2\|_2^2$ for some $\rho>0$. 
\end{definition}
Now we are ready to show the monotonicity of \eqref{mf-omi} under Assumption \ref{outstanding_assumptions_2}. Since $\{x|Ax=b,x\geq 0\}$ is a convex, closed and non-empty set, the normal cone operator $\mathcal{N}_{\{x|Ax=b,x\geq 0\}}$ is always monotone \cite{ryu2016primer}. Now note that by the definition of $c(L)$ in \eqref{kkt_param1} (as a concatenation of flattened vectors of negative rewards), we see that Assumption \ref{outstanding_assumptions_2} holds  if and only if $(c(L_1)-c(L_2))^\top (L_1-L_2)\geq \lambda\|L_1-L_2\|_2^2$ for any $L_1,\,L_2\in(\Delta(\mathcal{S}\times\mathcal{A}))^{T}$, \ie, $c(L)$ is monotone (and $\lambda$-strongly monotone if $\lambda>0$) on $(\Delta(\mathcal{S}\times\mathcal{A}))^{T}$.

Moreover, for any $\epsilon>0$, let us define a perturbed reward $\hat{r}_t^\epsilon(s,a,L_t):=r_t(s,a,L_t)-\epsilon L_t(s,a)$ ($s\in\mathcal{S},a\in\mathcal{A},t\in\mathcal{T},L_t\in\Delta(\mathcal{S}\times\mathcal{A})$), with the expected rewards being $\hat{R}^{\epsilon}=\{\hat{R}_t^{\epsilon}\}_{t\in\mathcal{T}}$. Then $c_{\hat{R}^\epsilon}(L)=c(L)+\epsilon L$, and since 
\[
(c_{\hat{R}^\epsilon}(L_1)-c_{\hat{R}^\epsilon}(L_2))^\top (L_1-L_2)=(c(L_1)-c(L_2))^\top (L_1-L_2)+\epsilon\|L_1-L_2\|_2^2,
\]
 we have that Assumption \ref{outstanding_assumptions_2} holds if and only if $(c_{\hat{R}^\epsilon}(L_1)-c_{\hat{R}^\epsilon}(L_2))^\top (L_1-L_2)\geq (\lambda+\epsilon)\|L_1-L_2\|_2^2$ for any $L_1,\,L_2\in(\Delta(\mathcal{S}\times\mathcal{A}))^{T}$, \ie, 
$c_{\hat{R}^\epsilon}(L)$ is $(\lambda+\epsilon)$-strongly monotone on  $(\Delta(\mathcal{S}\times\mathcal{A}))^{T}$. 

This implies that one can perturb the reward of any monotone model with a negative linear term to obtain a strongly monotone model. This observation is essential for the design of the algorithms to tackle MFGs that are not strongly monotone.

\subsection{Solving MF-OMI with forward-backward splitting (FBS)}
We now introduce computational algorithms for solving MF-OMI, which also serve as the basis for the learning algorithm in the next section. There are numerous operator splitting algorithms that can be adopted to solve monotone inclusion problems in the form of MF-OMI. Here we choose FBS (Forward Backward Splitting) \cite[\S 7.1]{ryu2016primer}, a simple yet efficient workhorse algorithm that is essentially a generalization of the projected descent algorithm to inclusion problems. 
FBS on such an inclusion problem proceeds as follows: in each iteration $k$,
\begin{equation}\label{FBS-prototype}
d^{k+1}=\texttt{Proj}_{\{x|Ax=b,x\geq0\}}(d^k-\alpha c(d^k)), 
\end{equation}
where $\alpha>0$ is the step-size for which appropriate ranges are specified below. Unfortunately, in general merely the monotonicity of $c$ is insufficient to guarantee convergence of FBS, and strong monotonicity is needed \cite[\S 7.1]{ryu2016primer}. Motivated by the discussion at the end of the previous section, we consider $\eta$-perturbed rewards for some perturbation coefficient $\eta>0$. Expanding \eqref{FBS-prototype} for MF-OMI with such perturbations, we obtain Algorithm \ref{MF-OMI-FBS-Cons}, which solves \eqref{mf-omi} via FBS with projection onto the set of occupation measures/consistency, \ie, $\{x|Ax=b,x\geq 0\}$. Note that the consistency projection step in Algorithm \ref{MF-OMI-FBS-Cons} is a convex quadratic program and can hence be solved efficiently via ADMM (\eg, via popular convex QP solvers such as OSQP \cite{stellato2020osqp}) and Frank-Wolfe \cite{jaggi2013revisiting} (where each linearized sub-problem can be exactly solved as an MDP) based algorithms. 


\begin{algorithm}[ht]
\caption{\texttt{MF-OMI-FBS}: MF-OMI with Forward-Backward Splitting}
\label{MF-OMI-FBS-Cons}
\begin{algorithmic}[1]
\STATE {\bfseries Input:} initial policy sequence $\pi^0\in\Pi$,
step-size $\alpha>0$, and perturbation coefficient $\eta\geq 0$.
\STATE Compute $d^0=L^{\pi^0}$.
\FOR{$k=0, 1, \dots$}
\STATE Update $\tilde{d}^{k+1}=d^k-\alpha (c(d^k)+\eta d^k)$.
\STATE Compute $d^{k+1}$ as the solution to the convex quadratic program: 
\[
\begin{array}{llll}
\text{minimize} & \|d-\tilde{d}^{k+1}\|_2^2 & \text{subject to}& Ad=b,\,d\geq 0.
\end{array}
\]\vspace{-0.57cm}
\ENDFOR
\end{algorithmic}
\end{algorithm}

\subsection{Convergence analysis of the \texttt{MF-OMI-FBS} algorithm}
Below we analyze the convergence of Algorithm \ref{MF-OMI-FBS-Cons} in terms of exploitability, which, together with Theorem \ref{n2mfg} leads to convergence rates of \texttt{NashConv} in the original $N$-player game. To account for the perturbation caused by $\eta$ regularization, we first prove the following result stating the Lipschitz continuity of exploitability as well as its robustness to reward perturbation. To facilitate the presentation,  we denote the exploitability of a policy sequence $\pi\in \Pi$ for an MFG with perturbed expected rewards $\hat{R}=\{\hat{R}_t\}_{t\in\mathcal{T}}$ as $\texttt{Expl}(\pi;\hat{R})$ (when everything else including transitions, etc. is fixed to the true model associated with the original $N$-player game). When $\hat{R}=\{\hat{R}_t\}_{t\in\mathcal{T}}=\{R_t\}_{t\in\mathcal{T}}=R$, \ie, the rewards are also fixed to the true model, we still use $\texttt{Expl}(\pi)$ to represent the exploitability.  


\begin{lemma}\label{lipschitz_and_reward_perturb_of_mfg}
Suppose that Assumption \ref{outstanding_assumptions_1} holds. Then for any $d^1,d^2\in\{x|Ax=b,x\geq 0\}$ and any policy sequences $\pi^1\in\texttt{Normalize}(d^1)$ and $\pi^2\in\texttt{Normalize}(d^2)$, then we have that for any $\epsilon\geq 0$, 
\[
\left|\texttt{Expl}(\pi^1;\hat{R}^\epsilon)-\texttt{Expl}(\pi^2;\hat{R}^\epsilon)\right|\leq (2TC_R+R_{\max}+(2T+1)\epsilon) \|d^1-d^2\|_1,
\]
where $\hat{R}^{\epsilon}=\{\hat{R}_t^{\epsilon}\}_{t\in\mathcal{T}}$ is the expectation of the perturbed reward $\hat{r}_t^\epsilon(s,a,L_t):=r_t(s,a,L_t)-\epsilon L_t(s,a)$ ($s\in\mathcal{S},a\in\mathcal{A},t\in\mathcal{T},L_t\in\Delta(\mathcal{S}\times\mathcal{A})$). 

In addition, for any $\epsilon\geq 0$ and any policy sequence $\pi\in\Pi$, we have 
\[
\left|\texttt{Expl}(\pi)-\texttt{Expl}(\pi;\hat{R}^\epsilon)\right|\leq 2T\epsilon.
\]
\end{lemma}


In order to prove Lemma \ref{lipschitz_and_reward_perturb_of_mfg}, the following lemma is needed, which demonstrates the dual transformation between policies and occupation measures. The proof can be found in Section \ref{sec:proof-lemma-consistency_recover}.
\begin{lemma}\label{consistency_recover}
Let $\hat{P}=\{\hat{P}_t\}_{t\in\mathcal{T}}$ be an arbitrary transition model, namely $\hat{P}_t(s'|s,a)\geq 0$ and $\sum_{s'\in\mathcal{S}}\hat{P}_t(s'|s,a)=1$. Suppose that $x\in\mathbb{R}^{SAT}$ is such that $A_{\hat{P}}x=b,x\geq 0$. Then for any $\pi\in\texttt{Normalize}(x)$, we have $\Gamma(\pi;\hat{P})=x$. In addition, for any $\pi\in\Pi$, we also have $A_{\hat{P}}\Gamma(\pi;\hat{P})=b$, $\Gamma(\pi;\hat{P})\geq 0$.
\end{lemma}

We are now ready to prove Lemma \ref{lipschitz_and_reward_perturb_of_mfg}.
\begin{proof}[Proof of Lemma \ref{lipschitz_and_reward_perturb_of_mfg}]
Firstly, by Lemma \ref{consistency_recover} (applied to $\hat{P}=P$, in which case $\Gamma(\pi;\hat{P})=L^{\pi}$), we have $L^{\pi^1}=d^1$ and $L^{\pi^2}=d^2$. 
Note that
\[
V^{\pi'}(L^{\pi})=\sum_{s\in\mathcal{S},a\in\mathcal{A},t\in\mathcal{T}}\hat{R}^{\epsilon}(s,a,L_t^{\pi})L_t^{\pi'}(s,a)=-(c_{\hat{R}^{\epsilon}}(L^{\pi}))^\top L^{\pi'}.
\]
Therefore for $i=1,2$, 
\[
\texttt{Expl}(\pi^i;\hat{R}^\epsilon)=\max_{\pi'\in\Pi}\left(-(c_{\hat{R}^{\epsilon}}(L^{\pi^i}))^\top L^{\pi'}\right)+(c_{\hat{R}^{\epsilon}}(L^{\pi^i}))^\top L^{\pi^i}=\max_{\pi'\in\Pi}\left(-(c_{\hat{R}^{\epsilon}}(d^i))^\top L^{\pi'}\right)+(c_{\hat{R}^{\epsilon}}(d^i))^\top d^i,
\]
which implies
\begin{align*}
&\left|\texttt{Expl}(\pi^1;\hat{R}^\epsilon)-\texttt{Expl}(\pi^2;\hat{R}^\epsilon)\right|\\
&\leq \max_{\pi'\in\Pi}\left|-\left(c_{\hat{R}^{\epsilon}}(d^1)-c_{\hat{R}^{\epsilon}}(d^2)\right)^\top L^{\pi'}\right|+\left|(c_{\hat{R}^{\epsilon}}(d^1))^\top (d^1-d^2)\right|+\left|(c_{\hat{R}^{\epsilon}}(d^1)-c_{\hat{R}^{\epsilon}}(d^2))^\top d^2\right|\\
&\leq T\left\|c_{\hat{R}^{\epsilon}}(d^1)-c_{\hat{R}^{\epsilon}}(d^2)\right\|_\infty+(R_{\max}+\epsilon)\|d^1-d^2\|_1+T\left\|c_{\hat{R}^{\epsilon}}(d^1)-c_{\hat{R}^{\epsilon}}(d^2)\right\|_\infty\\
&\leq (2TC_R+R_{\max}+(2T+1)\epsilon)\|d^1-d^2\|_1. 
\end{align*}
Finally, we also have 
\begin{align*}
&\left|\texttt{Expl}(\pi)-\texttt{Expl}(\pi;\hat{R}^{\epsilon})\right|\\
&=\left|\max_{\pi'\in\Pi}\left(-(c(L^{\pi}))^\top L^{\pi'}\right)+(c(L^{\pi}))^\top L^{\pi}-\max_{\pi'\in\Pi}\left(-(c_{\hat{R}^{\epsilon}}(L^{\pi}))^\top L^{\pi'}\right)-(c_{\hat{R}^{\epsilon}}(L^{\pi}))^\top L^{\pi}\right|\\
&\leq \max_{\pi'\in\Pi}\left|-\left(c(L^{\pi})-c_{\hat{R}^{\epsilon}}(L^\pi)\right)^\top L^{\pi'}\right|+\left|\left(c(L^{\pi})-c_{\hat{R}^{\epsilon}}(L^{\pi})\right)^\top L^{\pi}\right|\leq 2T\epsilon. 
\end{align*}
This completes the proof.
\end{proof}








The following lemma is key to the proof of the convergence of Algorithm \ref{MF-OMI-FBS-Cons}, which states the equivalence between the fixed-points of the \texttt{MF-OMI-FBS} iterations and the NEs of the MFG. Let $F_{\alpha,\eta}$ be the  mapping from $d^k$ to $d^{k+1}$ in Algorithm \ref{MF-OMI-FBS-Cons}, namely $F_{\alpha,\eta}(d)=\texttt{Proj}_{\{x|Ax=b,x\geq 0\}}(d-\alpha (c(d)+\eta d))$.  The proof of the following lemma can be found in Section \ref{sec:proof-lemma-MFNE_VS_FP}.
\begin{lemma}\label{MFNE_vs_FP}
For any $\eta\geq 0$, the set of fixed points of $F_{\alpha,\eta}$ is independent of $\alpha>0$. 
Now suppose that $\alpha>0$ and $\eta\geq 0$. Moreover,  
If $\pi\in\Pi$ is an NE of the MFG with the $\eta$-perturbed rewards $\hat{r}_t^\eta(s,a,L_t)=r_t(s,a,L_t)-\eta L_t(s,a)$, then $d=L^{\pi}$ is a fixed point of $F_{\alpha,\eta}$, namely $F_{\alpha,\eta}(d)=d$. Finally, if $d$ is a fixed point of $F_{\alpha,\eta}$, then $Ad=b,d\geq 0$ and any $\pi\in\texttt{Normalize}(d)$ is an NE of the MFG with the $\eta$-perturbed rewards $\hat{r}_t^\eta(s,a,L_t)$, and $d=L^\pi$. 
\end{lemma}

We are now ready to state the main convergence result of the proposed \texttt{MF-OMI-FBS} algorithm. 
\begin{theorem}\label{FBS_convergence}
Suppose that Assumptions \ref{outstanding_assumptions_1} and \ref{outstanding_assumptions_2} hold. Then we have the following convergence results. 
\begin{itemize}
    \item When $\lambda=0$ in Assumption \ref{outstanding_assumptions_2}, for any $\epsilon>0$, if we adopt the step-size $\alpha=\epsilon/(2C_R^2S^2A^2+2\epsilon^2)$ and the perturbation coefficient $\eta=\epsilon$, then for any $\pi^k\in\texttt{Normalize}(d^k)$ with $d^k$ from Algorithm \ref{MF-OMI-FBS-Cons}, we have 
\[
\texttt{Expl}(\pi^k)\leq 2T\epsilon+ 2\sqrt{SAT}(2{\color{black}T^2}C_R+R_{\max}{\color{black}T})\left(1-\kappa_{\epsilon}\right)^{\frac{k}{2}}, 
\]
where $\kappa_{\epsilon}:=\epsilon^2/\left(2C_R^2S^2A^2+2\epsilon^2\right)\in(0,1)$. 
\item When $\lambda>0$ in Assumption \ref{outstanding_assumptions_2}, if we adopt $\alpha=\lambda/(2C_R^2S^2A^2)$ and $\eta=0$, then for any $\pi^k\in\texttt{Normalize}(d^k)$ with $d^k$ from Algorithm \ref{MF-OMI-FBS-Cons}, we have  
\[
\texttt{Expl}(\pi^k)\leq 2\sqrt{SAT}(2{\color{black}T^2}C_R+R_{\max}{\color{black}T})\left(1-\kappa\right)^{\frac{k}{2}}, 
\]
where $\kappa:=\frac{\lambda^2}{2C_R^2S^2A^2}\in(0,1)$.
\end{itemize}
\end{theorem}
\begin{proof}
We prove the convergence result for generic $\lambda$ and $\eta$ with $\lambda+\eta>0$ and $\lambda\neq \eta$, and then specialize it to $\lambda=0,\,\eta>0$ and $\lambda>0,\,\eta=0$, resp. to derive the claimed conclusions.

Firstly, by Assumption \ref{outstanding_assumptions_1}, we have that for any $d^1,d^2\in(\Delta(\mathcal{S}\times\mathcal{A}))^T$, 
\begin{align}\label{lipschitz_cd}
\|c(d^1)-c(d^2)\|_2&=\sqrt{\sum_{s\in\mathcal{S},a\in\mathcal{A},t\in\mathcal{T}}(R_t(s,a,d_t^1)-R_t(s,a,d_t^2))^2} \notag\\
&\leq \sqrt{\sum_{s\in\mathcal{S},a\in\mathcal{A},t\in\mathcal{T}}C_R^2\|d_t^1-d_t^2\|_1^2}= C_R\sqrt{SA}\sqrt{\sum_{t\in\mathcal{T}}\|d_t^1-d_t^2\|_1^2} \\
&\leq C_R\sqrt{SA}\sqrt{\sum_{t\in\mathcal{T}}SA\|d_t^1-d_t^2\|_2^2}\leq C_RSA\|d^1-d^2\|_2. \notag
\end{align}


Now let $d^\star$ be a fixed-point of $F_{\alpha,\eta}$\footnote{Note that $d^\star$ in general varies as $\eta$ changes, but here we are considering a fixed $\eta$ and so we choose not to make the dependency explicit in the notation to keep it simple.}, which exists by Lemma \ref{MFNE_vs_FP} and the existence of mean-field NE (as guaranteed by Assumption \ref{outstanding_assumptions_1}). 
Then we have that for any $d\in (\Delta(\mathcal{S}\times\mathcal{A}))^{T}$
\begin{align}\label{F_contraction}
\left\|F_{\alpha,\eta}(d)-F_{\alpha,\eta}(d^\star)\right\|_2^2&\overset{(a)}{\leq} \left\|d-d^\star-\alpha (c(d)-c(d^\star)+\eta(d-d^\star))\right\|_2^2\notag\\
&=\|d-d^\star\|_2^2+\alpha^2\|c(d)-c(d^\star)+\eta(d-d^\star))\|_2^2\notag\\
&\qquad-2\alpha\left(c(d)-c(d^\star)+\eta\left(d-d^\star)\right)\right)^\top (d-d^\star)\\
&\overset{(b)}{\leq} \left(1+\alpha^2\left(2C_R^2S^2A^2+2\eta^2\right)\right)\|d-d^\star\|_2^2-2\alpha(\lambda+\eta)\|d-d^\star\|_2^2\notag\\
&=\left(1-2\alpha(\lambda+\eta)+2\alpha^2\left(C_R^2S^2A^2+\eta^2\right)\right)\|d-d^\star\|_2^2,\notag
\end{align}
where $(a)$ is by the non-expansiveness of projections onto closed convex sets, and $(b)$ is by \eqref{lipschitz_cd} and Assumption \ref{outstanding_assumptions_2}. Note that \eqref{F_contraction} indeed holds when $d^\star$ is an arbitrary mean-field flow in $(\Delta(\mathcal{S}\times\mathcal{A}))^{T}$.  

Hence for $\alpha=(\lambda+\eta)/(2C_R^2S^2A^2+2\eta^2)$, we have that \footnote{Note that this also implies that for such an $\alpha$ and when $\lambda+\eta>0$ and $\lambda\neq \eta$, $F_{\alpha,\eta}$ is a contraction mapping on $(\Delta(\mathcal{S}\times\mathcal{A}))^T$ and hence the fixed point $d^\star$ is unique.} \begin{equation}\label{F_alpha_eta_contractive}
\|F_{\alpha,\eta}(d)-F_{\alpha,\eta}(d^\star)\|_2\leq \sqrt{(1-(\lambda+\eta)^2/(2C_R^2S^2A^2+2\eta^2))}\|d-d^\star\|_2.
\end{equation}
Note that for any $d^1,d^2\in(\Delta(\mathcal{S}\times\mathcal{A}))^T$, 
\[
\lambda\|d^1-d^2\|_2^2\leq (c(d^1)-c(d^2))^\top (d^1-d^2)\leq \|c(d^1)-c(d^2)\|_2\|d^1-d^2\|_2\leq C_RSA\|d^1-d^2\|_2^2,
\]
which implies that $\lambda\leq C_RSA$. In addition, $\lambda\neq \eta$ implies 
\[
(\lambda+\eta)^2<2\lambda^2+2\eta^2\leq 2C_R^2S^2A^2+2\eta^2.
\]
Therefore
$0<(\lambda+\eta)^2/(2C_R^2S^2A^2+2\eta^2)<1$ whenever $\lambda\neq \eta$.

Hence we have 
\begin{align*}
\|d^k-d^\star\|_2&=\|F_{\alpha,\eta}^{(k)}(d^0)-F_{\alpha,\eta}^{(k)}(d^\star)\|_2\\
&\leq \left(1-(\lambda+\eta)^2/\left(2C_R^2S^2A^2+2\eta^2\right)\right)^{\frac{k}{2}}\|d^0-d^\star\|_2\leq 2{\color{black}T}\left(1-(\lambda+\eta)^2/\left(2C_R^2S^2A^2+2\eta^2\right)\right)^{\frac{k}{2}},
\end{align*}
where $F_{\alpha,\eta}^{(k)}:=\overbrace{F_{\alpha,\eta}\circ F_{\alpha,\eta}\circ \dots\circ F_{\alpha,\eta}}^{\text{$k$ times}}(d)$ denotes the self composition of $F_{\alpha,\eta}$ by $k$ times, and we make use of the fact that $F_{\alpha,\eta}(d^\star)=d^\star$.  {\color{black}Note that here we made use of the simple fact that 
\[
\|d^0-d^\star\|_2\leq \|d^0-d^\star\|_1=\sum\nolimits_{t\in\mathcal{T}}\|d^0_t-d^\star_t\|_1\leq \sum\nolimits_{t\in\mathcal{T}}(\|d^0_t\|_1+\|d^\star_t\|_1)= 2T.
\]
}

Finally, for any $\pi^k\in\texttt{Normalize}(d^k)$, let $\pi^\star\in\texttt{Normalize}(d^\star)$, then by Lemma \ref{MFNE_vs_FP} we have $\texttt{Expl}(\pi^\star;\hat{R}^\eta)=0$. Therefore Lemma \ref{lipschitz_and_reward_perturb_of_mfg} implies
\begin{align*}
\texttt{Expl}(\pi^k)&\leq |\texttt{Expl}(\pi^k)-\texttt{Expl}(\pi^\star)| +|\texttt{Expl}(\pi^\star)-\texttt{Expl}(\pi^\star;\hat{R}^\eta)|\\
&\leq (2TC_R+R_{\max})\|d^k-d^\star\|_1+2T\eta\\
&\leq 2T\eta+ 2\sqrt{SAT}(2{\color{black}T^2}C_R+R_{\max}{\color{black}{T}})\left(1-(\lambda+\eta)^2/\left(2C_R^2S^2A^2+2\eta^2\right)\right)^{\frac{k}{2}}.
\end{align*}
Finally, the proof is complete by taking $\eta=\epsilon>0$ when $\lambda=0$ and $\eta=0$ when $\lambda>0$.
\end{proof}




The following corollary shows the iteration complexity of Algorithm \ref{MF-OMI-FBS-Cons} for achieving a target tolerance of exploitability in a more explicit manner.  Particularly, we will see that to achieve an $\epsilon$ exploitability, (non-strong) Lasry-Lions monotonicity has a polynomial iteration complexity of $O(\epsilon^{-2}\log(1/\epsilon))$ and requires taking the target tolerance $\epsilon$ as an input for the algorithm parameter choices, while strong Lasry-Lions monotonicity leads to a logarithmic iteration complexity of $O(\log(1/\epsilon))$ without needing to specify an input target tolerance. The proof can be found in Section \ref{sec:proof-coro-FBS_iteration_complexity}. 
\begin{corollary}\label{FBS_iteration_complexity}
Suppose that Assumptions \ref{outstanding_assumptions_1} and \ref{outstanding_assumptions_2} hold. Then we have the following iteration complexities. 
\begin{itemize}
    \item When $\lambda=0$ in Assumption \ref{outstanding_assumptions_2}, for any target tolerance $\epsilon>0$, if we adopt $\alpha=\epsilon/(8C_R^2S^2A^2T+\epsilon^2/(2T))$ and  $\eta=\epsilon/(4T)$, then for any $k=\Omega(\epsilon^{-2}\log(1/\epsilon))$,\footnote{Here the big-$\Omega$ notation hides problem-dependent constants $S,A,T,C_R,R_{\max}$ and is in the sense of $\epsilon\rightarrow0$.} we have $\texttt{Expl}(\pi^k)\leq \epsilon$ for any $\pi^k\in\texttt{Normalize}(d^k)$ with $d^k$ from Algorithm \ref{MF-OMI-FBS-Cons}. 
    \item When $\lambda>0$ in Assumption \ref{outstanding_assumptions_2}, if we adopt $\alpha=\frac{\lambda}{2C_R^2S^2A^2}$ and $\eta=0$, then for any $k=\Omega(\log(1/\epsilon))$, we have $\texttt{Expl}(\pi^k)\leq \epsilon$ for any $\pi^k\in\texttt{Normalize}(d^k)$ with $d^k$ from Algorithm \ref{MF-OMI-FBS-Cons}.
    \end{itemize}
\end{corollary}


\section{\texttt{MF-OML}: Online mean-field RL for Nash equilibria}\label{sec:mf-oml}
In this section, we consider the episodic online reinforcement learning setup where the $N$ agents interact with each other and the environment repeatedly over episodes, without knowing the model. Each episode consists of the $T$-step $N$-player game defined at the beginning of Section \ref{setup_section}. 
We propose \texttt{MF-OML} ({\bf M}ean-{\bf F}ield {\bf O}ccupation-{\bf M}easure {\bf L}earning), an online mean-field RL algorithm for finding approximate NEs of symmetric $N$-player games based on the \eqref{mf-omi} formulation and the \texttt{MF-OMI-FBS} algorithm proposed in the previous section.

\paragraph{Nash regret.} 
To measure the performance of the online RL algorithm producing a sequence of strategy profiles $\boldsymbol{\pi}^m\in\Pi$ for each episode $m\geq 0$, we define a quantity $\texttt{NashRegret}(M)$ to characterize the cumulative deviation from Nash equilibrium of the algorithm up to episode $M-1$ ($M\geq 1$), which is formally defined as the following: 
\begin{equation}\label{nashregret}
\texttt{NashRegret}(M):=\sum_{m=0}^{M-1}\texttt{NashConv}(\boldsymbol{\pi}^m).
\end{equation}

\paragraph{Overview of \texttt{MF-OML}.} The design of \texttt{MF-OML} follows three steps. In the first step, in Section \ref{unreachable_modification}, we show that we can replace rewards $R$ and transitions $P$ at unreachable states (to be formally defined below) with arbitrarily specified default rewards (\eg, zero) and transitions (\eg, self-only transition), without changing the value of $\texttt{NashConv}$ at any strategy profile $\boldsymbol{\pi}$. This allows us to equivalently consider learning and solving the modified $N$-player game. This modification makes the model identifiable and thus is important in the learning procedure.  
We refer to the modified rewards and  transitions as $\tilde{R}$ and $\tilde{P}$ for later reference. 

In the second step, we extend 
\texttt{MF-OMI-FBS} to \texttt{MF-OMI-FBS-Approx} in Section \ref{conv_with_approx} to allow for inexact estimations of $c_{\tilde{R}}(d^k)$ and $\tilde{P}$ in each iteration $k\geq 0$. We provide the convergence analysis of this algorithm for any given estimation errors. Such an algorithm serves as a foundation for the final RL algorithm, where we provide concrete approximation procedures with statistical estimation guarantees. 

Then in the third step, we design an exploration scheme (\cf Algorithm \ref{alg:SampleExplore}) in Section \ref{explore_and_est} for the agents to learn the rewards and the transitions in each iteration $k\geq0$. 
To this end, in each iteration $k$ of \texttt{MF-OMI-FBS-Approx}, we conduct $n_k$ rounds of sample collections, each time picking one agent uniformly at random for pure exploration, while having the remaining $N-1$ agents executing the policy $\pi^k\in\texttt{Normalize}(d^k)$ induced from the current iteration. We establish the high probability bounds for the estimation errors of this exploration scheme. 

The final \texttt{MF-OML} algorithm, summarized in Algorithm \ref{MF-OML}, combines \texttt{MF-OMI-FBS-Approx} with the exploration and estimation procedures in the third step. We combine all the analyses in previous discussions and establish the regret bound for \texttt{MF-OML} in Section \ref{sec:regret}. 

To facilitate the presentation below, we use $\texttt{NashConv}(\boldsymbol{\pi};\tilde{R},\tilde{P})$ to denote the $\texttt{NashConv}$ of strategy profile $\boldsymbol{\pi}$ under the $N$-player game with modified rewards $\tilde{R}$ and transitions $\tilde{P}$. Similarly, we use $\texttt{Expl}(\pi;\tilde{R},\tilde{P})$ to denote the $\texttt{Expl}$ of policy $\pi$ under the MFG with modified rewards $\tilde{R}$ and transitions $\tilde{P}$. 




\subsection{Uniform exploration and default modifications at unreachable states}\label{unreachable_modification}
We begin by showing that modifying rewards and transitions at states that are unreachable under the uniform exploration policy sequence with default values will lead to equivalent games in the sense of $\texttt{NashConv}$. We first make the following definitions. 
\begin{definition}
    The uniform/pure exploration policy sequence $\pi^{\texttt{exp}}\in\Pi$ is defined such that $\pi_t^{\texttt{exp}}(s,a)=1/A$ ($s\in\mathcal{S}$, $a\in\mathcal{A}$, $t\in\mathcal{T}$). The unreachable states set $\mathcal{S}_t^{\texttt{unreach}}\subset\mathcal{S}$ is defined as  the subset of states that are not reachable at time step $t\in\mathcal{T}$ under the uniform policy $\pi^{\texttt{exp}}$, namely $\mathcal{S}_t^{\texttt{unreach}}=\{s\in\mathcal{S}|\mathbb{P}^{\pi^{\texttt{exp}}}(s_t=s)=0\}$, where $\mathbb{P}^{\pi}$ denotes the agent state-action trajectory probability induced by following policy sequence $\pi$. 
\end{definition}
Now we are ready to state the equivalence lemma for the modified models. 
\begin{lemma}\label{Pmod_equivalence} 
Define a modified transition model $\tilde{P}$ such that for any $a\in\mathcal{A}$,  $\tilde{P}_t(\cdot|s,a):=P_t(\cdot|s,a)$ when $s\notin \mathcal{S}_t^{\texttt{unreach}}$, while $\tilde{P}_t(\cdot|s,a)$ is set to $p_0(\cdot)$, an arbitrary probability vector in $\Delta(\mathcal{S})$ when $s\in\mathcal{S}_t^{\texttt{unreach}}$. Also define a modified reward $\tilde{R}_t(s,a,L_t)$ which is $\tilde{R}_t(s,a,L_t)=R_t(s,a,L_t)$ for $s\notin\mathcal{S}_t^{\texttt{unreach}}$ and  $\tilde{R}_t(s,a,L_t)=0$ otherwise. Then for any $\pi\in\Pi$, we have $\Gamma(\pi;P)=\Gamma(\pi;\tilde{P})$ and $\texttt{Expl}(\pi)=\texttt{Expl}(\pi;R,P)=\texttt{Expl}(\pi;\tilde{R},\tilde{P})$.  
In addition, for any strategy profile $\boldsymbol{\pi}$, we have $\texttt{NashConv}(\boldsymbol{\pi})=\texttt{NashConv}(\boldsymbol{\pi};R,P)=\texttt{NashConv}(\boldsymbol{\pi};\tilde{R},\tilde{P})$. 
\end{lemma}
\begin{proof}
Let $d_t$ be the occupation measure of policy $\pi$ under the original transition model $P$ (for an arbitrary agent $i$), namely $d=\Gamma(\pi;P)$. We first show that $s\notin \mathcal{S}_t^{\texttt{unreach}}$ for any $s$ with $d_t(s,a)>0$ for some $a\in\mathcal{A}$. We prove this by induction. Since $s\notin \mathcal{S}_0^{\texttt{unreach}}$ if and only if $s_0=s$, and  $d_0(s,a)>0$ implies $s_0=s$, the statement holds when $t=0$. Suppose the statement is true for time step $t$. If $s\in\mathcal{S}_{t+1}^{\texttt{unreach}}$ and $d_{t+1}(s,a)>0$ for some $a\in\mathcal{A}$, then $\mathbb{P}^{\pi}(s_{t+1}=s)=\sum_{a'\in\mathcal{A}}d_{t+1}(s,a')>0$. On the other hand, note that
\begin{align*}
    \mathbb{P}^{\pi^{\texttt{exp}}}(s_{t+1}=s)&=\sum_{s'\in\mathcal{S}}\sum_{a'\in\mathcal{A}}\mathbb{P}^{\pi^{\texttt{exp}}}(s_{t}=s')\pi^{\texttt{exp}}(a'|s')P_t(s|s',a').
\end{align*}
Since $\pi^{\texttt{exp}}(a'|s')>0$ for all $s'\in\mathcal{S}, a'\in\mathcal{A}$, $s\in\mathcal{S}_{t+1}^{\texttt{unreach}}$ implies $\mathbb{P}^{\pi^{\texttt{exp}}}(s_{t}=s')P_t(s|s',a')=0$ for all $s'\in\mathcal{S}, a'\in\mathcal{A}$, which further implies $P_t(s|s',a')=0$ for any $s'\notin \mathcal{S}_{t}^{\texttt{unreach}}$.

Then since
\begin{align*}
    \mathbb{P}^{\pi}(s_{t+1}=s)&=\sum_{s'\in\mathcal{S}}\sum_{a'\in\mathcal{A}}d_t(s',a')P_t(s|s',a')>0,
\end{align*}
there exists at least one pair of $s',a'\in\mathcal{S}\times\mathcal{A}$ such that $d_t(s',a')>0$ and $P_t(s|s',a')>0$. By induction, $d_t(s',a')>0$ implies $s'\notin \mathcal{S}_{t}^{\texttt{unreach}}$, which, together with $s\in\mathcal{S}_{t+1}^{\texttt{unreach}}$ and $P_t(s|s',a')>0$ leads to contradiction. Then we have shown that $s\notin \mathcal{S}_{t+1}^{\texttt{unreach}}$ for any $s$ such that $d_{t+1}(s,a)>0$ for some $a\in\mathcal{A}$, thus the induction is finished.

We then prove that for any player, the same policy $\pi$ induces the same occupation measures on theses two models, namely $d_t(s,a)=\tilde{d}_t(s,a)$ for all $s\in\mathcal{S},\,a\in\mathcal{A},\,t\in\mathcal{T}$, where $\tilde{d}=\Gamma(\pi;\tilde{P})$ is the occupation measure of policy $\pi$ under the modified transition model $\tilde{P}$ (for an arbitrary agent $i$). We show this by induction on time $t$. When $t=0$, it is true by the same initialization. Suppose $d_t(s,a)=\tilde d_t(s,a)$ for all $(s,a)\in\mathcal{S}\times\mathcal{A}$, where $d_t$ and $\tilde d_t$ are the occupation measures under the original model and the modified model, respectively. Since
\begin{align*}
    &d_{t+1}(s,a)=\pi_{t+1}(a|s)\mathbb{P}^{\pi}(s_{t+1}=s),\\
    &\tilde d_{t+1}(s,a)=\pi_{t+1}(a|s)\tilde{\mathbb{P}}^{\pi}(s_{t+1}=s),
\end{align*}
We only need to show $\mathbb{P}^{\pi}(s_{t+1}=s)=\tilde{\mathbb{P}}^{\pi}(s_{t+1}=s)$. In fact,
\begin{align*}
    \mathbb{P}^{\pi}(s_{t+1}=s)&=\sum_{s'\in\mathcal{S}}\sum_{a\in\mathcal{A}}d_t(s',a)P_t(s|s',a)=\sum_{s'\in\mathcal{S}}\sum_{a\in\mathcal{A}}\tilde d_t(s',a)P_t(s|s',a)\\
   &=\sum_{s'\in\mathcal{S}}\sum_{a\in\mathcal{A}}\tilde d_t(s',a)\tilde{P}_t(s|s',a)=\tilde{\mathbb{P}}^{\pi}(s_{t+1}=s),
\end{align*}
where we use the fact that $s'\notin \mathcal{S}_t^{\texttt{unreach}}$ for any $s'$ with $d_t(s',a)>0$ and $P_t(s|s',a)=\tilde{P}_t(s|s',a)$ for $s'\notin \mathcal{S}_t^{\texttt{unreach}}$. The induction is finished (and hence we have proved that $\Gamma(\pi;P)=d=\tilde{d}=\Gamma(\pi;\tilde{P})$).

The last step is to prove the $\texttt{NashConv}$ of the same strategy profile under these two $N$-player games are the same. This reduces to showing that for any $i\in[N]$,  $V^i(\boldsymbol{\pi})=\tilde V^i(\boldsymbol{\pi})$ for any strategy profile $\boldsymbol{\pi}$. Here $\tilde V^i(\boldsymbol{\pi})=\tilde{\mathbb{E}}_{\boldsymbol{\pi}}\left[\sum_{t\in\mathcal{T}}\tilde r_t(s_t^i,a_t^i,L_t^N)\right]$ is the expected cumulative reward of agent $i$ under the modified model. More precisely, the expectation $\tilde{\mathbb{E}}_{\boldsymbol{\pi}}$ is over the trajectory of states and actions when the agents take independent actions $a_t^j\sim \pi_t^j(s_t^j)$ for $j\in[N]$ and $t\in\mathcal{T}$ under the modified rewards $\tilde{R}$ and transitions $\tilde{P}$. Let $d_t^j$ (resp. $\tilde{d}_t^j$) be the occupation measure of agent $j\in[N]$ under the original (resp. modified) model, who takes policy sequence $\pi^j$. 
Then we have $d_t^j(s,a)=\tilde{d}_t^j(s,a)$ for any $j\in[N],s\in\mathcal{S},a\in\mathcal{A}$ by the previous induction proof. In addition, we have
\begin{align*}
    \tilde V^i(\boldsymbol{\pi})=\tilde{\mathbb{E}}_{\boldsymbol{\pi}}\left[\sum_{t\in\mathcal{T}}\tilde r_t(s_t^i,a_t^i,L_t^N)\right]=\sum_{t\in\mathcal{T}}\sum_{({\bf s}_t,{\bf a}_t)\in\mathcal{S}^N\times\mathcal{A}^N}\tilde{R}_t(s_t^i,a_t^i,L^N({\bf s}_t,{\bf a}_t))\prod_{j\in[N]}d^j_t(s_t^j,a_t^j),
\end{align*}
where $L^N$ maps $\mathcal{S}^N\times\mathcal{A}^N$ to $\Delta(\mathcal{S}\times\mathcal{A})$, defined as  
\begin{align*}
    L^N({\bf s},{\bf a})(s,a):=\frac{1}{N}\sum_{i\in[N]}{\bf 1}(s^i=s,a^i=a).
\end{align*}
Since we have proved that $d_t^i(s_t^i,a_t^i)>0$ implies $s_t^i\notin\mathcal{S}_t^{\texttt{unreach}}$, therefore by definition we also have that $\tilde{R}_t(s_t^i,a_t^i,L^N({\bf s}_t,{\bf a}_t))=R_t(s_t^i,a_t^i,L^N({\bf s}_t,{\bf a}_t))$ when $d_t^i(s_t^i,a_t^i)>0$. Hence 
we have 
\begin{align*}
    \tilde V^i(\boldsymbol{\pi})&=\sum_{t\in\mathcal{T}}\sum_{({\bf s}_t,{\bf a}_t)\in\mathcal{S}^N\times\mathcal{A}^N}\tilde{R}_t(s_t^i,a_t^i,L^N({\bf s}_t,{\bf a}_t))\prod_{j\in[N]}\tilde{d}^j_t(s_t^j,a_t^j)\\
    &=\sum_{t\in\mathcal{T}}\sum_{({\bf s}_t,{\bf a}_t)\in\mathcal{S}^N\times\mathcal{A}^N}R_t(s_t^i,a_t^i,L^N({\bf s}_t,{\bf a}_t))\prod_{j\in[N]}d^j_t(s_t^j,a_t^j)=V^i(\boldsymbol{\pi}).
\end{align*}
The proof for exploitability is nearly identical and hence omitted. 
The proof is thus finished.
\end{proof}


\subsection{Convergence analysis with approximations}\label{conv_with_approx} 


We now propose Algorithm \ref{MF-OMI-FBS-Cons-Approx}, namely \texttt{MF-OMI-FBS-Approx}, which is an inexact version of \texttt{MF-OMI-FBS} with approximation oracles $\hat{c}^k$ and $\hat{P}^k$ for the reward vector $c_{\tilde{R}}(d^k)$ and the transition model $\tilde{P}$ with default modifications defined in Lemma \ref{Pmod_equivalence}, respectively. 
\begin{algorithm}[ht]
\caption{\texttt{MF-OMI-FBS-Approx}: MF-OMI-FBS with Approximation Oracles}
\label{MF-OMI-FBS-Cons-Approx}
\begin{algorithmic}[1]
\STATE {\bfseries Input:} initial policy sequence $\pi^0\in\Pi$,
step-size $\alpha>0$, and perturbation coefficient $\eta\geq 0$.
\STATE Compute $d^0=L^{\pi^0}$.
\FOR{$k=0, 1, \dots$}
\STATE Compute an approximation $\hat{c}^k$ of $c_{\tilde{R}}(d^k)$ and an approximate transition model $\hat{P}^k$ of $\tilde{P}$, resp. 
\STATE Update $\tilde{d}^{k+1}=d^k-\alpha (\hat{c}^k+\eta d^k)$.
\STATE Compute $d^{k+1}$ as the solution to the convex quadratic program: 
\[
\begin{array}{llll}
\text{minimize} & \|d-\tilde{d}^{k+1}\|_2^2 & \text{subject to}& A_{\hat{P}^k}d=b,\,d\geq 0.
\end{array}
\]\vspace{-0.57cm}
\ENDFOR
\end{algorithmic}
\end{algorithm}

The following theorem establishes the convergence of \texttt{MF-OMI-FBS-Approx}. 
\begin{theorem}\label{FBS_convergence_approx}
Suppose that Assumptions \ref{outstanding_assumptions_1} and \ref{outstanding_assumptions_2} hold, 
and in addition that for $k\geq 0$, 
\[
\|\hat{c}^k-c_{\tilde{R}}(d^k)\|_2\leq \epsilon_1^k,\quad  \max_{s\in\mathcal{S},a\in\mathcal{A},t=0,\dots,T-2}\sum_{s'\in\mathcal{S}}|\hat{P}_t^k(s'|s,a)-\tilde{P}_t(s'|s,a)|\leq \epsilon_2^k,
\]
where {\color{black}$\epsilon_1^k,\epsilon_2^k\geq 0$}, $\hat{c}^k\in\mathbb{R}^{SAT}$ with $\|\hat{c}^k\|_{\infty}\leq R_{\max}$, $\hat{P}^k=\{\hat{P}_t\}_{t\in\mathcal{T}}$ is a transition model, namely $\hat{P}_t^k(s'|s,a)\geq 0$ and $\sum_{s'\in\mathcal{S}}\hat{P}_t^k(s'|s,a)=1$. Then we have the following convergence results. 
\begin{itemize}
\item When $\lambda=0$ in Assumption \ref{outstanding_assumptions_2}, for any $\epsilon>0$, if we adopt the step-size $\alpha=\epsilon/(2C_R^2S^2A^2+2\epsilon^2)$ and the perturbation coefficient $\eta=\epsilon$, then for any $\pi^k\in\texttt{Normalize}(d^k)$ with $d^k$ from Algorithm \ref{MF-OMI-FBS-Cons-Approx}, we have 
\[
\texttt{Expl}(\pi^k)\leq 2T\epsilon+ \sqrt{SAT}(2{\color{black}T^2}C_R+R_{\max}{\color{black}T})\left(2\left(1-\kappa_{\epsilon}\right)^{\frac{k}{2}}+\sum_{j=0}^{k-1}(1-\kappa_{\epsilon})^{\frac{k-j-1}{2}}\tilde{\epsilon}^j\right), 
\]
where $\kappa_{\epsilon}:=\epsilon^2/(2C_R^2S^2A^2+2\epsilon^2)\in(0,1)$ and \[
\tilde{\epsilon}^j:=\dfrac{T(T-1)}{2}\epsilon_2^j+T\sqrt{(4+2\alpha\eta+2\alpha R_{\max}SA)(T-1)\epsilon_2^j}+\alpha\epsilon_1^j.
\]
\item When $\lambda>0$ in Assumption \ref{outstanding_assumptions_2}, if we adopt  $\alpha=\lambda/(2C_R^2S^2A^2)$ and  $\eta=0$, then for any $\pi^k\in\texttt{Normalize}(d^k)$ with $d^k$ from Algorithm \ref{MF-OMI-FBS-Cons-Approx}, we have 
\[
\texttt{Expl}(\pi^k)\leq \sqrt{SAT}(2{\color{black}T^2}C_R+R_{\max}{\color{black}T})\left(2\left(1-\kappa\right)^{\frac{k}{2}}+\sum_{j=0}^{k-1}(1-\kappa)^{\frac{k-j-1}{2}}\tilde{\epsilon}^j\right), 
\]
where $\kappa:=\lambda^2/(2C_R^2S^2A^2)\in(0,1)$, and $\tilde{\epsilon}^j$ is defined the same as above when $\lambda=0$. 
\end{itemize}
\end{theorem}
We will need the following lemmas to prove Theorem \ref{FBS_convergence_approx}, the proof of which can be found in Section \ref{sec:proof-lemma-projection_d_vs_L}.
\begin{lemma}\label{projection_d_vs_L}
Let $\hat{P}$ be an arbitrary transition model, and let $\tilde{d}\in \mathbb{R}^{SAT}$ be some arbitrary constant vector. Let $\hat{d}$ be the solution to minimizing $\|d-\tilde{d}\|_2^2$ over the set of $d$ with $A_{\hat{P}}d=b,d\geq 0$. In addition, let $\hat{\pi}$ be a solution to minimizing $\|\Gamma(\pi;\hat{P})-\tilde{d}\|_2$ over $\hat{\pi}\in\Pi$. Then $\hat{d}=\Gamma(\hat{\pi};\hat{P})$. 
\end{lemma}

\begin{lemma}\label{lipschitz_gamma_pi_P}
Let $\hat{P}$ be an arbitrary transition model. Then we have that for any $\pi\in\Pi$, 
\[
\|\Gamma(\pi;\hat{P})-\Gamma(\pi;\tilde{P})\|_1\leq \dfrac{T(T-1)}{2}\max_{s\in\mathcal{S},a\in\mathcal{A},t=0,\dots,T-2}\sum_{s'\in\mathcal{S}}|\hat{P}_t(s'|s,a)-\tilde{P}_t(s'|s,a)|.
\]
\end{lemma}

\begin{proof}[Proof of Theorem \ref{FBS_convergence_approx}]
As in the proof of Theorem \ref{FBS_convergence}, we prove the convergence result for generic $\lambda$ and $\eta$ with $\lambda+\eta>0$ and $\lambda\neq \eta$, and then specialize it to $\lambda=0,\,\eta>0$ and $\lambda>0,\,\eta=0$, resp. to derive the claimed conclusions.

Firstly, notice that by Lemma \ref{projection_d_vs_L}, the projection step \[
d^{k+1}=\texttt{Proj}_{\{x|A_{\hat{P}^k}x=b,{\color{black}x\geq 0}\}}(\tilde{d}^{k+1})=\texttt{Proj}_{\{x|A_{\hat{P}^k}x=b,x\geq 0\}}(d^k-\alpha(\hat{c}^k+\eta d^k))
\]
on Line 6 of Algorithm \ref{MF-OMI-FBS-Cons-Approx} can be denoted as $d^{k+1}=\Gamma(\pi^{k+1};\hat{P}^k)$, where $\pi^{k+1}$ minimizes $\|\Gamma(\pi;\hat{P}^k)-\tilde{d}^{k+1}\|_2^2$ over $\pi\in\Pi$. 


Now we show that $d^{k+1}$ is close to 
\[
d^{k+1,\star}:=\texttt{Proj}_{\{x|A_{\tilde{P}}x=b,x\geq 0\}}(\tilde{d}^{k+1})=\texttt{Proj}_{\{x|A_{\tilde{P}}x=b,x\geq 0\}}(d^k-\alpha(\hat{c}^k+\eta d^k)).
\]
Again by Lemma \ref{projection_d_vs_L}, we have $d^{k+1,\star}=\Gamma(\pi^{k+1,\star};\tilde{P})$ where $\pi^{k+1,\star}$ minimizes $\|\Gamma(\pi;\tilde{P})-\tilde{d}^{k+1}\|_2^2$ over $\pi\in\Pi$. In addition, we have that for any $\pi\in\Pi$, by Lemma \ref{lipschitz_gamma_pi_P}, we have
\[
\left|\|\Gamma(\pi;\tilde{P})-\tilde{d}^{k+1}\|_2-\|\Gamma(\pi;\hat{P}^k)-\tilde{d}^{k+1}\|_2\right|\leq \|\Gamma(\pi;\tilde{P})-\Gamma(\pi;\hat{P}^k)\|_2\leq \|\Gamma(\pi;\tilde{P})-\Gamma(\pi;\hat{P}^k)\|_1\leq\dfrac{T(T-1)}{2}\epsilon_2^k. 
\]
Hence we have 
\begin{align}\label{dd_diff_of_diff_bound}
&\left|\|d^{k+1,\star}-\tilde{d}^{k+1}\|_2-\|d^{k+1}-\tilde{d}^{k+1}\|_2\right|=\left|\inf_{\pi\in\Pi}\|\Gamma(\pi;P)-\tilde{d}^{k+1}\|_2-\inf_{\pi\in\Pi}\|\Gamma(\pi;\hat{P}^k)-\tilde{d}^{k+1}\|_2\right| \notag\\
&=\left|-\sup_{\pi\in\Pi}\left(-\|\Gamma(\pi;P)-\tilde{d}^{k+1}\|_2\right)+\sup_{\pi\in\Pi}\left(-\|\Gamma(\pi;\hat{P}^k)-\tilde{d}^{k+1}\|_2\right|\right)\\
&\leq \sup_{\pi\in\Pi}\left|\|\Gamma(\pi;P)-\tilde{d}^{k+1}\|_2-\|\Gamma(\pi;\hat{P}^k)-\tilde{d}^{k+1}\|_2\right|\leq \dfrac{T(T-1)}{2}\epsilon_2^k. \notag
\end{align}


Now let $\pi^{k+1}\in\texttt{Normalize}(d^{k+1})$ and then define $\hat{d}^{k+1,\star}:=\Gamma(\pi^{k+1};P)$. Then noticing that $d^{k+1}=\Gamma(\pi^{k+1};\hat{P}^k)$ by Lemma \ref{consistency_recover}, we have by Lemma \ref{lipschitz_gamma_pi_P} that 
\begin{equation}\label{dd_diff_bound}
\|d^{k+1}-\hat{d}^{k+1,\star}\|_2\leq\dfrac{T(T-1)}{2}\epsilon_2^k.
\end{equation}

Next we show that $\hat{d}^{k+1,\star}$ is feasible and $O(\epsilon_2^k)$-sub-optimal for minimizing $\|d-\tilde{d}^{k+1}\|_2^2$ subject to $Ad=b,d\geq 0$. To see this, first notice that by Lemma \ref{consistency_recover}, we have $A\hat{d}^{k+1,\star}=b$ and $\hat{d}^{k+1,\star}\geq 0$, and hence $\hat{d}^{k+1,\star}$ is feasible for the aforementioned quadratic optimization problem. Then to show the sub-optimality gap, we simply notice that by \eqref{dd_diff_of_diff_bound} and \eqref{dd_diff_bound}, 
\begin{align}\label{final_dd_diff_of_diff_bound}
&\|\hat{d}^{k+1,\star}-\tilde{d}^{k+1}\|_2\leq  \|d^{k+1} - \tilde{d}^{k+1}\|_2+\|d^{k+1}-\hat{d}^{k+1,\star}\|_2\notag\\
&\leq \left|\|d^{k+1}-\tilde{d}^{k+1}\|_2-\|d^{k+1,\star}-\tilde{d}^{k+1}\|_2\right|+\|d^{k+1,\star}-\tilde{d}^{k+1}\|_2+\|d^{k+1}-\hat{d}^{k+1,\star}\|_2\\
&\leq \|d^{k+1,\star}-\tilde{d}^{k+1}\|_2 + T(T-1)\epsilon_2^k.\notag
\end{align}

Then by the strong convexity of the quadratic objective $f(d):=\|d-\tilde{d}^{k+1}\|_2^2$, we have 
\begin{align}\label{sharpness_bound}
f(\hat{d}^{k+1,\star})-f(d^{k+1,\star})&\geq \nabla f(d^{k+1,\star})^\top ( \hat{d}^{k+1,\star}-d^{k+1,\star})+
\|\hat{d}^{k+1,\star}-d^{k+1,\star}\|_2^2\\
&\geq \|\hat{d}^{k+1,\star}-d^{k+1,\star}\|_2^2,\notag
\end{align}
where the first inequality is by strong convexity, and for the second inequality, we make use of the property that $d^{k+1,\star}$ is the optimal solution to minimizing $f(d)$ subject to $Ad=b,d\geq 0$, and that $\hat{d}^{k+1,\star}$ is also feasible in the sense that $A\hat{d}^{k+1,\star}=b,\hat{d}^{k+1,\star}\geq 0$, and hence by the global optimality condition of convex constrained optimization problems \cite[\S4.2.3]{boyd2004convex} we have $\nabla f(d^{k+1,\star})^\top ( \hat{d}^{k+1,\star}-d^{k+1,\star})\geq 0$.

Now since $d^k$, $\hat{d}_t^{k+1,\star}$ and $d_t^{k+1,\star}$ ($t\in\mathcal{T}$) are probability distributions over $\mathcal{S}\times\mathcal{A}$ due to Lemma \ref{consistency_recover} and their definitions by $\Gamma$, and since $\tilde{d}^{k+1}=d^k-\alpha(\hat{c}^k+\eta d^k)$, 
we have
\begin{align*}
&\|\hat{d}^{k+1,\star}-\tilde{d}^{k+1}\|_2+\|d^{k+1,\star}-\tilde{d}^{k+1}\|_2\leq \|\hat{d}^{k+1,\star}-\tilde{d}^{k+1}\|_1+\|d^{k+1,\star}-\tilde{d}^{k+1}\|_1\\
&\leq \|\hat{d}^{k+1,\star}-(1-\alpha\eta)d^k\|_1+\|d^{k+1,\star}-(1-\alpha\eta)d^k\|_1+2\alpha\|\hat{c}^k\|_1\\
&\leq 2(2+\alpha\eta)T+2\alpha R_{\max}SAT=(4+2\alpha\eta+2\alpha R_{\max}SA)T.
\end{align*}
Hence we have that by \eqref{final_dd_diff_of_diff_bound} and \eqref{sharpness_bound}, 
\begin{align*}
\|\hat{d}^{k+1,\star}-d^{k+1,\star}\|_2^2&\leq \|\hat{d}^{k+1,\star}-\tilde{d}^{k+1}\|_2^2-\|d^{k+1,\star}-\tilde{d}^{k+1}\|_2^2\\
&=(\|\hat{d}^{k+1,\star}-\tilde{d}^{k+1}\|_2+\|d^{k+1,\star}-\tilde{d}^{k+1}\|_2)(\|\hat{d}^{k+1,\star}-\tilde{d}^{k+1}\|_2-\|d^{k+1,\star}-\tilde{d}^{k+1}\|_2)\\
&\leq (4+2\alpha\eta+2\alpha R_{\max}SA)T^2(T-1)\epsilon_2^k, 
\end{align*}
and hence together with \eqref{dd_diff_bound}, we have 
\begin{equation}\label{projection_approx_error}
\begin{split}
\|d^{k+1}-d^{k+1,\star}\|_2&\leq \|d^{k+1}-\hat{d}^{k+1,\star}\|_2+\|\hat{d}^{k+1,\star}-d^{k+1,\star}\|_2\\
&\leq \dfrac{T(T-1)}{2}\epsilon_2^k+T\sqrt{(4+2\alpha\eta+2\alpha R_{\max}SA)(T-1)\epsilon_2^k}.
\end{split}
\end{equation}


Define $\hat{F}^k_{\alpha,\eta}(d):=\texttt{Proj}_{\{x|A_{\hat{P}^k}x=b,x\geq0\}}(d-\alpha(\hat{c}^k+\eta d))$. Then Algorithm \ref{MF-OMI-FBS-Cons-Approx} can be represented compactly as $d^{k+1}=\hat{F}^k_{\alpha,\eta}(d^k)$. Now let $d^\star$ be a fixed-point of $F_{\alpha,\eta}$, which exists by Lemma \ref{MFNE_vs_FP} (applied to $\tilde{R}$ and $\tilde{P}$) and the existence of mean-field NE (as guaranteed by Assumption \ref{outstanding_assumptions_1}). Then we have
\begin{align*}
 &\|d^{k+1}-d^\star\|_2 = \left\|\hat{F}^k_{\alpha,\eta}(d^k)-F_{\alpha,\eta}(d^\star) \right\|_2 \\
 &\leq \left\|\texttt{Proj}_{\{x|A_{\hat{P}^k}x=b,x\geq 0\}}(d^k-\alpha(\hat{c}^k+\eta d^k))-\texttt{Proj}_{\{x|A_{\tilde{P}}x=b,x\geq 0\}}(d^k-\alpha(\hat{c}^k+\eta d^k))\right\|_2\\
 &\quad+\left\|\texttt{Proj}_{\{x|A_{\tilde{P}}x=b,x\geq 0\}}(d^k-\alpha(\hat{c}^k+\eta d^k))-\texttt{Proj}_{\{x|A_{\tilde{P}}x=b,x\geq 0\}}(d^k-\alpha(c(d^k)+\eta d^k))\right\|_2 \\
 &\quad +\left\|\texttt{Proj}_{\{x|A_{\tilde{P}}x=b,x\geq 0\}}(d^k-\alpha(c(d^k)+\eta d^k))-F_{\alpha,\eta}(d^\star)\right\|_2\\
 &= \|d^{k+1}-d^{k+1,\star}\|_2+\|F_{\alpha,\eta}(d^k)-F_{\alpha,\eta}(d^\star)\|_2\\
 &\quad +\left\|\texttt{Proj}_{\{x|A_{\tilde{P}}x=b,x\geq 0\}}(d^k-\alpha(\hat{c}^k+\eta d^k))-\texttt{Proj}_{\{x|A_{\tilde{P}}x=b,x\geq 0\}}(d^k-\alpha(c(d^k)+\eta d^k))\right\|_2\\
 &\leq \dfrac{T(T-1)}{2}\epsilon_2^k+T\sqrt{(4+2\alpha\eta+2\alpha R_{\max}SA)(T-1)\epsilon_2^k}\\
 &\quad+\sqrt{(1-(\lambda+\eta)^2/(2C_R^2S^2A^2+2\eta^2))}\|d^k-d^\star\|_2+\alpha\epsilon_1^k,
\end{align*}
where the last step makes use of the contraction inequality \eqref{F_alpha_eta_contractive} of $F_{\alpha,\eta}$ proved in Theorem \ref{FBS_convergence} (applied to $\tilde{R}$ and $\tilde{P}$ here, which also satisfy Assumption \ref{outstanding_assumptions_1} and Assumption \ref{outstanding_assumptions_2} by the definitions of the modifications), \eqref{projection_approx_error} and the non-expansive property of projections. 

By telescoping over $k$, we conclude that 
\begin{align*}
\|d^k-d^\star\|_2&\leq (1-\kappa_{\lambda,\eta})^{\frac{k}{2}}\|d^0-d^\star\|_2+\sum_{j=0}^{k-1}(1-\kappa_{\lambda,\eta})^{\frac{k-j-1}{2}}\tilde{\epsilon}^j\\
&\leq 2(1-\kappa_{\lambda,\eta})^{\frac{k}{2}}+\sum_{j=0}^{k-1}(1-\kappa_{\lambda,\eta})^{\frac{k-j-1}{2}}\tilde{\epsilon}^j,
\end{align*}
where $\kappa_{\lambda,\eta}:=(\lambda+\eta)^2/(2C_R^2S^2A^2+2\eta^2)\in(0,1)$ and \[
\tilde{\epsilon}^j:=\dfrac{T(T-1)}{2}\epsilon_2^j+T\sqrt{(4+2\alpha\eta+2\alpha R_{\max}SA)(T-1)\epsilon_2^j}+\alpha\epsilon_1^j.
\]

Finally, as in the end of the proof of Theorem \ref{FBS_convergence}, by defining $\hat{\tilde{R}}^{\eta}(s,a,L_t):=\tilde{R}_t(s,a,L_t)-\eta L_t(s,a)$, 
for any $\pi^k\in\texttt{Normalize}(d^k)$, let $\pi^\star\in\texttt{Normalize}(d^\star)$, then by Lemma \ref{MFNE_vs_FP} applied to $\tilde{R}$ and $\tilde{P}$, we have $\texttt{Expl}(\pi^\star;\hat{\tilde{R}}^\eta;\tilde{P})=0$, and hence by Lemma \ref{lipschitz_and_reward_perturb_of_mfg} applied to $\tilde{R},\tilde{P}$ and $\hat{\tilde{R}}^{\eta}$, we have 
\begin{equation}\label{FBS_convergence_approx_general}
\begin{split}
\texttt{Expl}(\pi^k;\tilde{R},\tilde{P})&\leq |\texttt{Expl}(\pi^k;\tilde{R},\tilde{P})-\texttt{Expl}(\pi^\star;\tilde{R},\tilde{P})| +|\texttt{Expl}(\pi^\star;\tilde{R},\tilde{P})-\texttt{Expl}(\pi^\star;\hat{\tilde{R}}^\eta,\tilde{P})|\\
&\leq (2TC_R+R_{\max})\|d^k-d^\star\|_1+2T\eta\\
&\leq 2T\eta+ \sqrt{SAT}(2{\color{black}T^2}C_R+R_{\max}{\color{black}T})\left(2\left(1-\kappa_{\lambda,\eta}\right)^{\frac{k}{2}}+\sum_{j=0}^{k-1}(1-\kappa_{\lambda,\eta})^{\frac{k-j-1}{2}}\tilde{\epsilon}^j\right).
\end{split}
\end{equation}
Finally, by noticing that $\texttt{Expl}(\pi^k)=\texttt{Expl}(\pi^k;R,P)=\texttt{Expl}(\pi^k;\tilde{R},\tilde{P})$ thanks to Lemma \ref{Pmod_equivalence}, the proof is complete by taking $\eta=\epsilon>0$ when $\lambda=0$ and $\eta=0$ when $\lambda>0$. 
\end{proof}

\subsection{Exploration and estimations}\label{explore_and_est}
We now introduce the third step towards the final online mean-field RL algorithm. In this section, we design a sampling and exploration scheme and the associated estimation procedures of rewards and transitions (with default modification) to compute estimates $\hat{c}^k$ and $\hat{P}^k$ and the associated estimation errors $\epsilon_1^k,\,\epsilon_2^k$ for Algorithm \ref{MF-OMI-FBS-Cons-Approx} in Theorem \ref{FBS_convergence_approx}. 

\paragraph{Sampling and exploration.} The main idea for exploration is to conduct $n_k$ independent rounds of sample collections in each iteration $k$ of Algorithm \ref{MF-OMI-FBS-Cons-Approx}, each time randomly selecting one agent for pure exploration, while having the remaining $N-1$ agents following the current policy $\pi^k\in\texttt{Normalize}(d^k)$. The algorithm for sampling and exploring at iteration $k$ is summarized in Algorithm \ref{alg:SampleExplore}.\footnote{A more precise notation for the trajectory states, actions and rewards would be $s_t^{i_0(k,l),k,l}$, $a_t^{i_0(k,l),k,l}$ and $r_t^{i_0(k,l),k,l}$, which would fully disambiguate the different episodes $l$ and iterations $k$. But since $k,l$ already appear in $i_0(k,l)$, we simplify them to be $s_t^{i_0(k,l)}$, $a_t^{i_0(k,l)}$ and $r_t^{i_0(k,l)}$ to facilitate the discussion below.\label{footnote_kl_notation}}

\begin{algorithm}[ht]
\caption{$\texttt{SampleExplore}(\pi^k,n_k,k)$}
\label{alg:SampleExplore}
\begin{algorithmic}[1]
\FOR{$l=0,1,\dots,n_k-1$ (independently)}
\STATE Sample agent $i_0(k,l)$ from $[N]$ uniformly. Let agent $i_0(k,l)$ take the exploration policy sequence $\pi_t^{\texttt{exp}}(s,a)=1/A$ ($s\in\mathcal{S}$, $a\in\mathcal{A}$, $t\in\mathcal{T}$), and the other $N-1$ agents all follow policy sequence $\pi^k$. Collect trajectory data $\{(s_t^{i_0(k,l)},a_t^{i_0(k,l)},r_t^{i_0(k,l)},s_{t+1}^{i_0(k,l)})\}_{t\in\mathcal{T}}$.
\ENDFOR 
\STATE \textbf{Output:} 
$\{(s_t^{i_0(k,l)},a_t^{i_0(k,l)},r_t^{i_0(k,l)},s_{t+1}^{i_0(k,l)})\}_{t\in\mathcal{T},l\in\{0,\dots,n_k-1\}}$.
\end{algorithmic}
\end{algorithm}

\paragraph{Estimations.} Given the collected trajectory data $\{(s_t^{i_0(j,l)},a_t^{i_0(j,l)},r_t^{i_0(j,l)},s_{t+1}^{i_0(j,l)})\}_{t\in\mathcal{T},l\in\{0,\dots,n_j-1\},j\leq k}$ from Algorithm \ref{alg:SampleExplore} up to iteration $k$, the estimated rewards and transition probabilities are then computed via (conditional) sample mean estimations based on the collected trajectories, with default values ($0$ for rewards and $p_0$ for transitions, as stated in Lemma \ref{Pmod_equivalence}) used for states and actions that are not visited by the trajectories.  More concretely, we compute estimated rewards and transitions at iteration $k$ as follows.

For any $(s,a,t)\in\mathcal{S}\times\mathcal{A}\times\mathcal{T}$,
\begin{equation}\label{R_est}
    \hat{R}_t^{k}(s,a)=
        \begin{cases}
            \dfrac{1}{n_{k}(s,a,t)}\sum_{l=0}^{n_k-1}r_t^{i_0(k,l)}{\bf 1}\{(s_t^{i_0(k,l)},a_t^{i_0(k,l)})=(s,a)\}, &\text{ if } n_k(s,a,t)>0,\\
            0,&\text{ if } n_k(s,a,t)=0.
        \end{cases}
    \end{equation}

For any $(s,a,s',t)\in\mathcal{S}\times\mathcal{A}\times\mathcal{S}\times\mathcal{T}$, 
    \begin{equation}\label{P_est}
    \hat{P}_t^k(s'|s,a))=
        \begin{cases}
            \dfrac{\sum_{j\leq k}\sum_{l=0}^{n_j-1}{\bf 1}\{(s_t^{i_0(j,l)},a_t^{i_0(j,l)},s_{t+1}^{i_0(j,l)})=(s,a,s')\}}{\sum_{j\leq k}n_j(s,a,t)}, &\text{ if } \sum_{j\leq k}n_j(s,a,t)>0,\\
            p_0(s'),&\text{ if } \sum_{j\leq k}n_j(s,a,t)=0.
        \end{cases}
    \end{equation}
    Here $n_k(s,a,t)=\sum_{l=0}^{n_k-1}{\bf 1}\{s_t^{i_0(k,l)}=s,a_t^{i_0(k,l)}=a\}$.

Note that since the transition models do not depend on the mean-field terms $d_t^k$ which vary over iterations, we collect all sample trajectories in the history (instead of only for iteration $k$ as in the case of rewards estimations) for the estimations of the transition models. 

Finally, we concatenate the estimations into 
\begin{equation}\label{c_and_P_est}
    \hat{c}^k=\left[
\begin{array}{c}
-\hat{R}_0^k(\cdot,\cdot)\\
\vdots\\
-\hat{R}_T^k(\cdot,\cdot)
\end{array}
\right]\in\mathbb{R}^{SAT} \text{ and } \hat{P}^k
\end{equation}
for use in iteration $k$ of Algorithm \ref{MF-OMI-FBS-Cons-Approx}.

\paragraph{{\color{black}Statistical} estimation errors.} 
The following proposition establishes the statistical estimation errors of the reward estimations \eqref{R_est} and transition estimations \eqref{P_est} for each state, action and time step. We leave a remainder term $C_R\|d_t^k-L_t^{\pi^k}\|_1$ with $L_t^{\pi^k}=\Gamma(\pi^k;P)$ to better illustrate the components of the reward estimation errors. The explicit bound of the estimation errors is given in Corollary \ref{cor:epsilon_1_and_epsilon_2}. 
\begin{proposition}\label{prop:RP_est_error}
Suppose that Assumption \ref{outstanding_assumptions_1} holds. 
    For any $s\in\mathcal{S},a\in\mathcal{A},t\in\mathcal{T}$ and any $\delta>0$, if $n_k> \frac{\log(2/\delta)}{2p_{\min}^2}$, then with probability at least $1-(1+S)\delta$, the following two bounds hold:
    \begin{equation}\label{RP_est_error}
    \begin{split}
        \left|\hat{R}_t^k(s,a)-\tilde{R}_t(s,a,d_t^k)\right|&\leq \underbrace{C_RSA\left(\frac{1}{N}+\sqrt{\frac{\pi}{2N}}\right)}_{\text{mean-field approximation error}}+\underbrace{\sqrt{\frac{2R_{\max}^2\log(4/\delta)}{p_{\min}n_k-\sqrt{\log(2/\delta)n_k/2}}}}_{\text{concentration error}}+\underbrace{C_R\|d_t^k-L_t^{\pi^k}\|_1}_{\text{execution error}},\\
        |\hat{P}_t^k(s'|s,a)-\tilde{P}_t(s'|s,a)|&\leq \underbrace{\sqrt{\frac{2\log(4/\delta)}{p_{\min}\sum_{j\leq k}n_j-\sqrt{\log(2/\delta)\sum_{j\leq k}n_j/2}}}}_{\text{concentration error}}, \text{ for all } s'\in\mathcal{S}. 
        \end{split}
    \end{equation}
    Here $L_t^{\pi^k}=\Gamma(\pi^k;P)$, $p_{\min}=\min_{s\notin \mathcal{S}_t^{\texttt{unreach}},\, t\in\mathcal{T}}d^{\pi^{\texttt{exp}}}_t(s,a)>0$, where $d^{\pi^{\texttt{exp}}}$ is the occupation measure under the pure exploration policy $\pi^{\texttt{exp}}$ with $\pi_t^{\texttt{exp}}(s,a)=1/A$ ($s\in\mathcal{S},\,a\in\mathcal{A},\,t\in\mathcal{T}$). 
\end{proposition}

\begin{remark}
    We briefly comment on the compositions of the reward estimation errors and transition estimation errors. Due to the dependency on the mean-field terms $d_t^k$, the reward estimation errors are combinations of mean-field approximation error similar to Theorem \ref{n2mfg}, concentration errors from the Hoeffding inequality, and execution errors resulting from the mapping from $d_t^k$ to $\pi^k$ via \texttt{Normalize}. The symmetry aggregation effect of the mean-field terms also necessitates novel conditioning techniques that are not needed in the classical MARL literature. In contrast, the transition estimation errors consists of merely the concentration errors since the transition model is independent of the mean-field terms.
\end{remark}



\begin{proof}

Let $L_t^{N,k,l}(s,a)=\dfrac{1}{N}\sum_{i\in[N]}{\bf 1}\{s_t^{i,k,l}=s,a_t^{i,k,l}=a\}$ ($s\in\mathcal{S},\,a\in\mathcal{A},\,t\in\mathcal{T}$), where $s_t^{i,k,l}$ and $a_t^{i,k,l}$ denote the state and action of agent $i\in[N]$ at time step $t\in\mathcal{T}$ in the $l$-th episode of trajectory collection in $\texttt{SampleExporeEstimation}(\pi^k,n_k,k)$. Particularly, recall from footnote \ref{footnote_kl_notation} that $s_t^{i_0(k,l),k,l}=s_t^{i_0(k,l)}$ and $a_t^{i_0(k,l),k,l}=a_t^{i_0(k,l)}$.  

\medskip

\noindent\textit{\textbf{Part 1: Gap between $\hat{R}_t^k(s,a)$ and $\tilde{R}_t(s,a,d_t^k)$.}} We first bound the gap between $\hat{R}_t^k(s,a)$ and $\tilde{R}_t(s,a,L_t^{\pi^k})$. When $s\in\mathcal{S}_t^{\texttt{unreach}}$, $n_k(s,a,t)=0$ a.s., and hence both $\hat{R}_t^k(s,a)$ and $\tilde{R}_t(s,a,L_t^\pi)$ are zero a.s.. Thus we only need to bound the gap for $s\notin \mathcal{S}_t^{\texttt{unreach}}$. 
    
    We first provide a high probability bound for $n_k(s,a,t)$. Since $n_k(s,a,t)=\sum_{l=0}^{n_k-1}{\bf 1}\{s_t^{i_0(k,l)}=s,a_t^{i_0(k,l)}=a\}$, and since ${\bf 1}\{s_t^{i_0(k,l)}=s,a_t^{i_0(k,l)}=a\}$ ($l=0,\dots,n_k-1$) are $n_k$ i.i.d. random variables with a.s. bound $1$ and expectation $d_t^{\texttt{exp}}(s,a)$,  
    by Hoeffding's inequality we have for any $s\notin \mathcal{S}_t^{\texttt{unreach}}$ and $\epsilon>0$,
    \begin{align*}
         \mathbb{P}\left(n_k(s,a,t)\geq (p_{\min}-\epsilon)n_k\right)&\geq \mathbb{P}\left(n_k(s,a,t)\geq (d_t^{\pi^{\texttt{exp}}}(s,a)-\epsilon)n_k\right)\\
         &\geq 1-\exp\left(-2n_k\epsilon^2\right).
    \end{align*}
    Therefore, for any $\delta>0$ and any $s\notin \mathcal{S}_t^{\texttt{unreach}}$, with probability at least $1-\delta/2$,  $$n_k(s,a,t)\geq \left(p_{\min}-\sqrt{\frac{\log(2/\delta)}{2n_k}}\right)n_k>0$$ 
    since $n_k> \log(2/\delta)/(2p_{\min}^2)$.
    
    When $n_k(s,a,t)>0$, let $I_t^k(s,a)=\{l\in\{0,\dots,n_k-1\}:(s_t^{i_0(k,l)},a_t^{i_0(k,l)})=(s,a)\}$. Note that 
        \begin{align*}
        &\left|\hat{R}_t^k(s,a)-\tilde{R}_t(s,a,L_t^{\pi^k})\right|=\left|\dfrac{1}{n_{k}(s,a,t)}\sum_{l\in I_t^k(s,a)}r_t^{i_0(k,l)}-\tilde{R}_t(s,a,L_t^{\pi^k})\right|\\
        &\leq \underbrace{\left|\dfrac{1}{n_{k}(s,a,t)}\sum_{l\in I_t^k(s,a)}\left(r_t^{i_0(k,l)}-\mathbb{E}\left[\tilde{R}_t(s,a,L_t^{N,k,l})|I_t^k(s,a)\right]\right)\right|}_{J_1}\\
        &\quad+\underbrace{\left|\dfrac{1}{n_{k}(s,a,t)}\sum_{l\in I_t^k(s,a)}\left(\mathbb{E}\left[\tilde{R}_t(s,a,L_t^{N,k,l})|I_t^k(s,a)\right]-\tilde{R}_t(s,a,L_t^{\pi^k})\right)\right|}_{J_2}\\
        &:=J_1+J_2.
    \end{align*}

    To bound $J_1$, notice that when $n_k(s,a,t)=|I_t^k(s,a)|>0$, conditioned on $I_t^k(s,a)$, we have that\footnote{This can be verified by the definition of independence via noticing that conditioning on $I_t^k(s,a)$ is equivalent to conditioning on $\{{\bf 1}\{(s_t^{i_0(k,l)},a_t^{i_0(k,l)})=(s,a)\}\}_{l=0}^{n_k-1}$. 
    } 
        $$r_t^{i_0(k,l)}-\mathbb{E}\left[\tilde{R}_t(s,a,L_t^{N,k,l})|I_t^k(s,a)\right] \quad(l\in I_t^k(s,a))$$ 
    are $n_k(s,a,t)$ bounded independent random variables with bound $2R_{\max}$ and conditional mean $0$. 
    Hence Hoeffding inequality implies that
    $$
    \mathbb{P}\left(J_1\geq \epsilon|I_t^k(s,a)\right)\leq 2\exp\left(-\frac{n_k(s,a,t)\epsilon^2}{2R_{\max}^2}\right),
    $$
    and thus
    \begin{align*}
    \mathbb{P}(J_1\geq \epsilon|n_k(s,a,t))&=\mathbb{E}\left[\mathbb{E}\left[{\bf 1}\{J_1\geq \epsilon\}|I_t^k(s,a)\right]{\color{black}\Big|}n_k(s,a,t)\right]\\
    &=\mathbb{E}\left[\mathbb{P}(J_1\geq\epsilon|I_t^k(s,a){\color{black})}|n_k(s,a,t)\right]\leq 2\exp(-n_k(s,a,t)\epsilon^2/(2R_{\max}^2)).
    \end{align*}
    
    Hence we have 
    \begin{align*}
        \mathbb{P}\left(J_1\geq \epsilon\right)&=\mathbb{P}\left(J_1\geq \epsilon,n_k(s,a,t)\geq\left(p_{\min}-\sqrt{\frac{\log(2/\delta)}{2n_k}}\right)n_k\right)\\
        &\qquad+\mathbb{P}\left(J_1\geq \epsilon,n_k(s,a,t)<\left(p_{\min}-\sqrt{\frac{\log(2/\delta)}{2n_k}}\right)n_k\right)\\
        &\leq\sum_{n=n_k^0}^\infty \mathbb{P}\left(J_1\geq \epsilon|n_k(s,a,t)=n\right)\mathbb{P}\left(n_k(s,a,t)=n\right)+\delta/2\\
        &\leq \sum_{n=n_k^0}^\infty 2\exp\left(-\frac{n\epsilon^2}{2R_{\max}^2}\right)\mathbb{P}\left(n_k(s,a,t)=n\right)+\delta/2 \\
        &\leq 2\exp\left(-\frac{n_k^0\epsilon^2}{2R_{\max}^2}\right)\mathbb{P}\left(n_k(s,a,t)\geq n_k^0\right)+\delta/2\\
        &\leq 2\exp\left(-\frac{\left(p_{\min}-\sqrt{\frac{\log(2/\delta)}{2n_k}}\right)n_k\epsilon^2}{2R_{\max}^2}\right)+\delta/2,
    \end{align*}  
    where $n_k^0=\lceil(p_{\min}-\sqrt{\frac{\log(2/\delta)}{2n_k}})n_k\rceil>0$. 
    Then with probability at least $1-\delta$, $J_1\leq \sqrt{\frac{2R_{\max}^2\log(4/\delta)}{p_{\min}n_k-\sqrt{\log(2/\delta)n_k/2}}}$.

    To bound $J_2$, we have that a.s.,
    \begin{align*}
        J_2&\leq \dfrac{1}{n_{k}(s,a,t)}\sum_{l\in I_t^k(s,a)}\left|\mathbb{E}\left[\tilde{R}_t(s,a,L_t^{N,k,l})\Big|I_t^k(s,a)\right]-\tilde{R}_t(s,a,L_t^{\pi^k})\right|\\
        &\leq\dfrac{C_R}{n_{k}(s,a,t)}\sum_{l\in I_t^k(s,a)}\mathbb{E}\left[\|L_t^{N,k,l}-L_t^{\pi^k}\|_1\Big|I_t^k(s,a)\right]
    \end{align*}

    Note that 
    \begin{align*}
        &\mathbb{E}\left[\left|L_t^{N,k,l}(s,a)-L_t^{\pi^k}(s,a)\right|\Big|I_t^k(s,a)\right]\\
        =&\mathbb{E}\left[\mathbb{E}\left[\left|\frac{1}{N}\sum\nolimits_{j\in[N]}{\bf 1}\{s_t^{j,k,l}=s,a_t^{j,k,l}=a\}-L_t^{\pi^k}(s,a)\right|\Big|I_t^k(s,a),i_0(k,l)\right]\Big|I_t^k(s,a)\right]\\
        \leq & \frac{1}{N}+\sqrt{\frac{\pi}{2N}}.
    \end{align*}
Here we reuse the trick in the proof of Theorem \ref{n2mfg}: Conditioned on $I_t^k(s,a)$ and $i_0(k,l)$, ${\bf 1}\{s_t^{j,k,l}=s,a_t^{j,k,l}=a\}$ ($j\in[N]\backslash\{i_0(k,l)\}$) are bounded independent random variables with 
\begin{align*}
&\left|\mathbb{E}\left[\frac{1}{N}\sum\nolimits_{j\in[N]\backslash\{i_0(k,l)\}}{\bf 1}\{s_t^{j,k,l}=s,a_t^{j,k,l}=a\}\Big|I_t^k(s,a),i_0(k,l)\right]-L_t^{\pi^k}(s,a)\right|\leq 1/N. 
\end{align*}
Hence again by Hoeffding inequality as in the proof of Theorem \ref{n2mfg}, we have that a.s., $$J_2\leq C_RSA\left(\frac{1}{N}+\sqrt{\frac{\pi}{2N}}\right).$$

Combining the upper bound of $J_1$ and $J_2$, we have with probability at least $1-\delta$, $$\left|\hat{R}_t^k(s,a)-\tilde{R}_t(s,a,L_t^{\pi^k})\right|\leq C_RSA\left(\frac{1}{N}+\sqrt{\frac{\pi}{2N}}\right)+\sqrt{\frac{2R_{\max}^2\log(4/\delta)}{p_{\min}n_k-\sqrt{\log(2/\delta)n_k/2}}}.$$

And finally, since $|\tilde{R}_t(s,a,L_t^{\pi^k}-\tilde{R}_t(s,a,d_t^k)|\leq C_R\|L_t^{\pi^k}-d_t^k\|_1$ by Assumption \ref{outstanding_assumptions_1}, the proof for the rewards estimation errors is complete.
    

\medskip

\noindent\textit{\textbf{Part 2: Gap between $\hat{P}_t^k(s'|s,a)$ and $\tilde{P}_t(s'|s,a)|$.}} Now we bound the gap $|\hat{P}_t^k(s'|s,a)-\tilde{P}_t(s'|s,a)|$. The major difference here is that since the transition models are mean-field independent, we can collect all samples in the history for the estimations. 
Similarly, if $s\in\mathcal{S}_t^{\texttt{unreach}}$, then we have a.s. $n_k(s,a,t)=0$, and hence $\hat{P}_t^k(s'|s,a)=p_0(s')=\tilde{P}_t(s'|s,a)$. Therefore we only need to consider $s\notin\mathcal{S}_t^{\texttt{unreach}}$.
We first obtain a similar high probability bound of $\sum_{j\leq k}n_j(s,a,t)$. Again by Hoeffding inequality, for any $\epsilon>0$
\begin{align*}
    \mathbb{P}\left(\sum_{j\leq k}n_j(s,a,t)\geq (p_{\min}-\epsilon)\sum_{j\leq k}n_j\right)&\geq \mathbb{P}\left(\sum_{j\leq k}n_j(s,a,t)\geq (d_t^{\pi^{\texttt{exp}}}(s,a)-\epsilon)\sum_{j\leq k}n_j\right)\\
         &\geq 1-\exp\left(-2\sum_{j\leq k}n_j\epsilon^2\right).
\end{align*}
Therefore, for any $\delta>0$ and any $s\notin \mathcal{S}_t^{\texttt{unreach}}$, with probability at least $1-\delta/2$, $$\sum_{j\leq k}n_j(s,a,t)\geq \tilde n_k^0:=\left(p_{\min}-\sqrt{\frac{\log(2/\delta)}{2\sum_{j\leq k}n_j}}\right)\sum_{j\leq k}n_j>0$$
since $\sum_{j\leq k}n_j\geq n_k>\log(2/\delta)/(2p_{\min}^2)$. 
Then when   $\sum_{j\leq k}n_j(s,a,t)>0$, for any $s'\in\mathcal{S}$
    \begin{align*}
        |\hat{P}_t^k(s'|s,a)-\tilde{P}_t(s'|s,a)|&=\left|\dfrac{1}{\sum_{j\leq k}n_j(s,a,t)}\sum_{j\leq k}\sum_{l=0}^{n_j-1}\left({\bf 1}\{(s_t^{i_0(j,l)},a_t^{i_0(j,l)},s_{t+1}^{i_0(j,l)})=(s,a,s')\}\right.\right.\\
        &\qquad\left.\left. -\tilde P_t(s'|s,a){\bf 1}\{(s_t^{i_0(j,l)},a_t^{i_0(j,l)})=(s,a)\}\right)\right|\\
        &=\left|\dfrac{1}{\sum_{j\leq k}n_j(s,a,t)}\sum_{(j,l)\in \tilde I^k_t(s,a)}\left({\bf 1}\{s_{t+1}^{i_0(j,l)}=s'\} -\tilde P_t(s'|s,a)\right)\right|,
    \end{align*}
where \[
\tilde I^k_t(s,a)=\{(j,l):j=1,\dots,k,l=0,\dots,n_j, (s_t^{i_0(j,l)},a_t^{i_0(j,l)})=(s,a)\}.
\]
Note that conditioned on $\tilde{I}_t^k(s,a)$, 
\[{\bf 1}\{s_{t+1}^{i_0(j,l)}=s'\}-P_t(s'|s,a),\quad (j,l)\in\tilde I_t^k(s,a)
\]
are $\sum_{j\leq k}n_j(s,a,t)$ independent random variables with values in $[-1,1]$ and conditional expectation 0. Then similar to the bound of $J_1$ above, using Hoeffding inequality one can get 
\begin{align*}
    \mathbb{P}\left(|\hat{P}_t^k(s'|s,a)-\tilde{P}_t(s'|s,a)|\geq \epsilon\Bigg|\sum_{j\leq k}n_j(s,a,t)\right)\leq 2\exp\left(-\frac{1}{2}\epsilon^2\sum_{j\leq k}n_j(s,a,t)\right).
\end{align*}

Hence similarly
\begin{align*}
    \mathbb{P}\left(|\hat{P}_t^k(s'|s,a)-\tilde{P}_t(s'|s,a)|\geq \epsilon\right)
    &=\mathbb{P}\left(|\hat{P}_t^k(s'|s,a)-\tilde{P}_t(s'|s,a)|\geq \epsilon,\sum_{j\leq k}n_j(s,a,t)\geq \tilde n_k^0\right)\\
    &\qquad +\mathbb{P}\left(|\hat{P}_t^k(s'|s,a)-\tilde{P}_t(s'|s,a)|\geq \epsilon,\sum_{j\leq k}n_j(s,a,t)< \tilde n_k^0\right)\\
    &\leq 2\exp\left(-\frac{\tilde n_k^0\epsilon^2}{2}\right)+\delta/2.
\end{align*}
Therefore with probability at least $1-\delta$, 
$$
|\hat{P}_t^k(s'|s,a)-\tilde{P}_t(s'|s,a)|\leq \sqrt{\frac{2\log(4/\delta)}{p_{\min}\sum_{j\leq k}n_j-\sqrt{\log(2/\delta)\sum_{j\leq k}n_j/2}}}.
$$
Combining the two statements finishes the proof.
\end{proof}

\begin{corollary}\label{cor:epsilon_1_and_epsilon_2}
    For each $k>0$ and $\delta>0$, if $n_k> \frac{2\log(2/\delta)}{p_{\min}^2}$, then with probability at least $1-(1+S)SAT\delta$, the sample estimations from \eqref{c_and_P_est} satisfy the following bounds: 
    \[
\|\hat{c}^k-c_{\tilde{R}}(d^k)\|_2\leq \epsilon_1^k,\quad  \max_{s\in\mathcal{S},a\in\mathcal{A},t=0,\dots,T-2}\sum_{s'\in\mathcal{S}}|\hat{P}_t^k(s'|s,a)-\tilde{P}_t(s'|s,a)|\leq \epsilon_2^k,
\]
where 
\begin{align}
    \epsilon_1^k=&\sqrt{SAT}\left(C_RSA\left(\frac{1}{N}+\sqrt{\frac{\pi}{2N}}\right)+2\sqrt{\frac{R_{\max}^2\log(4/\delta)}{p_{\min}n_k}}\right)+C_RS^2AT(T-1)\sqrt{\frac{\log(4/\delta)}{p_{\min}\sum_{j\leq k}n_j}},\label{eq:epsilon_1}\\
\epsilon_2^k=&2S\sqrt{\frac{\log(4/\delta)}{p_{\min}\sum_{j\leq k}n_j}}.\label{eq:epsilon_2}
\end{align}
\end{corollary}
\begin{proof}
By Proposition \ref{prop:RP_est_error}, for any $\delta>0$, if $n_k> \frac{\log(2/\delta)}{2p_{\min}^2}$, then with probability at least $1-(1+S)SAT\delta$, the inequalities in \eqref{RP_est_error} hold for all $s\in\mathcal{S}$,$a\in\mathcal{A}$, and $t\in\mathcal{T}$, which implies 

\begin{equation*}
\max_{s\in\mathcal{S},a\in\mathcal{A},t=0,\dots,T-2}\sum_{s'\in\mathcal{S}}|\hat{P}_t^k(s'|s,a)-\tilde{P}_t(s'|s,a)|
\leq S\sqrt{\frac{2\log(4/\delta)}{p_{\min}\sum_{j\leq k}n_j-\sqrt{\log(2/\delta)\sum_{j\leq k}n_j/2}}},
\end{equation*}
and 
\begin{equation*}
    \|\hat{c}^k-c_{\tilde{R}}(L^{\pi^k})\|_2\leq \sqrt{SAT}\left(C_RSA\left(\frac{1}{N}+\sqrt{\frac{\pi}{2N}}\right)+\sqrt{\frac{2R_{\max}^2\log(4/\delta)}{p_{\min}n_k-\sqrt{\log(2/\delta)n_k/2}}}\right).
\end{equation*}
In addition, when these inequalities hold, since Lemma \ref{Pmod_equivalence} implies that $L^{\pi^k}=\Gamma(\pi^k;P)=\Gamma(\pi^k;\tilde{P})$,  by Lemma \ref{consistency_recover} and Lemma \ref{lipschitz_gamma_pi_P}, 
\begin{align}\label{density_est_error}
\|d^k-L^{\pi^k}\|_1&\leq \dfrac{T(T-1)}{2}\max_{s\in\mathcal{S},a\in\mathcal{A},t=0,\dots,T-2}\sum_{s'\in\mathcal{S}}|\hat{P}_t^k(s'|s,a)-\tilde{P}_t(s'|s,a)|\\
&\leq \dfrac{ST(T-1)}{2}\sqrt{\frac{2\log(4/\delta)}{p_{\min}\sum_{j\leq k}n_j-\sqrt{\log(2/\delta)\sum_{j\leq k}n_j/2}}},\notag
\end{align}
thus by Lipschitz continuity,
\begin{align*}
    \|c_{\tilde{R}}(d^k)-c_{\tilde{R}}(L^{\pi^k})\|_2&\leq \|c_{\tilde{R}}(d^k)-c_{\tilde{R}}(L^{\pi^k})\|_1\leq C_RSA\|d^k-L^{\pi^k}\|_1\\
    &\leq\dfrac{C_RS^2AT(T-1)}{2}\sqrt{\frac{2\log(4/\delta)}{p_{\min}\sum_{j\leq k}n_j-\sqrt{\log(2/\delta)\sum_{j\leq k}n_j/2}}}.
\end{align*}

Combining all the statements above finishes the proof.


\end{proof}

\subsection{\texttt{MF-OML} and regret analysis}\label{sec:regret}
In this section, we put together the ingredients from the previous sections into the final online mean-field RL algorithm, \texttt{MF-OML}. The algorithm plugs Algorithm \ref{alg:SampleExplore} and the associated estimation procedures into the iterations of Algorithm \ref{MF-OMI-FBS-Cons-Approx} that is applied to the rewards and transitions with default modifications. The final algorithm is summarized as Algorithm \ref{MF-OML}. 
\begin{algorithm}[ht]
\caption{\texttt{MF-OML}: Single-Phase {\bf M}ean-{\bf F}ield {\bf O}ccupation-{\bf M}easure {\bf L}earning }
\label{MF-OML}
\begin{algorithmic}[1]
\STATE {\bfseries Input:}  $d^0\in(\Delta(\mathcal{S}\times\mathcal{A}))^T$,
 $\alpha>0$,  $\eta>0$, $\{n_k\}_{k\geq 0}$.
\STATE Compute $\pi^0\in\texttt{Normalize}(d^0)$. 
\FOR{$k=0, 1, \dots,$}
\STATE Collect $\{(s_t^{i_0(k,l)},a_t^{i_0(k,l)},r_t^{i_0(k,l)},s_{t+1}^{i_0(k,l)})\}_{t\in\mathcal{T},l\in\{0,\dots,n_k-1\}}=\texttt{SampleExplore}(\pi^k,n_k,k)$.
\STATE Compute estimated rewards and transitions from the collected data with \eqref{R_est} and \eqref{P_est}, and then construct $\hat{c}^k$ and $\hat{P}^k$ with \eqref{c_and_P_est}.   
\STATE Update $\tilde{d}^{k+1}=d^k-\alpha(\hat{c}^k+\eta d^k)$.
\STATE Compute $d^{k+1}$ as the solution to the convex quadratic program:
\[
\begin{array}{llll}
\text{minimize} & \|d-\tilde{d}^{k+1}\|_2^2 & \text{subject to}& A_{\hat{P}^k}d=b,\,d\geq 0.
\end{array}
\]\vspace{-0.45cm}
\STATE Extract policy $\pi^{k+1}\in\texttt{Normalize}(d^{k+1})$. 
\ENDFOR 
\end{algorithmic}
\end{algorithm}

The following theorem establishes the regret bounds of \texttt{MF-OML} for both strongly Lasry-Lions monotone and (non-strongly) Lasry-Lions monotone settings. 
\begin{theorem}\label{thm:mf_oml_regret}
    Suppose that Assumptions \ref{outstanding_assumptions_1} and \ref{outstanding_assumptions_2} hold. Then we have the following regret bounds for \texttt{MF-OML}.
    \begin{itemize}
        \item When $\lambda=0$, for a given number of episodes $M$, if we adopt  $\alpha=\eta/(2C_R^2S^2A^2+2\eta^2)$ with  $\eta=\max\{N^{-1/6},M^{-1/12}\}$ and $n_k=k^3$, then we have that with probability at least $1-\pi^2(1+S)SAT\delta/6$, \[
        \texttt{NashRegret}(M)=O\left(\frac{S^2A^2{\color{black}T^3}M}{N^{1/6}}+S^2A{\color{black}T^{9/2}}M^{11/12}(\log M)^{5/4}\sqrt{\log(1/\delta)}\right).
        \]
        \item When $\lambda>0$, if we adopt  $\alpha=\lambda/(2C_R^2S^2A^2)$, $\eta=0$, $n_k=k^3$, then we have that with probability at least $1-\pi^2(1+S)SAT\delta/6$, for all $M\geq 1$ we have \[
        \texttt{NashRegret}(M)=O\left(\frac{S^2A^2{\color{black}T^3}M}{\sqrt{N}}+S^2A{\color{black}T^{9/2}}M^{3/4}(\log M)^{5/4}\sqrt{\log(1/\delta)}\right).
        \]
    \end{itemize}
\end{theorem}
We need the following lemma to connect the strategy profile $\boldsymbol{\pi}^{k,l}=(\pi^k,\dots,\overbrace{\pi^{\texttt{exp}}}^{i_0(k,l)\text{-th}},\dots,\pi^k)$ executed in the $l$-th episode ($l=0,\dots,n_k-1$) of iteration $k$ of \texttt{MF-OML}  to the symmetric strategy profile $\boldsymbol{\pi}^{k}=(\pi^k,\dots,\pi^k)$. Here we fix/condition on iteration $k$ and the picked agent index $i_0(k,l)$ for the following lemma.
\begin{lemma}\label{lemma:deviation_one_policy_nashconv}
    Suppose that Assumption \ref{outstanding_assumptions_1} holds. Then we have \[
    |\texttt{NashConv}(\boldsymbol{\pi}^{k,l})-\texttt{Expl}(\pi^k)|\leq 2C_RSAT\sqrt{\dfrac{\pi}{2N}}+\dfrac{6C_RSAT}{N}.
    \]
\end{lemma}
The proof is mostly identical to the proof Theorem \ref{n2mfg} except for having at most two agents deviating from the policy $\pi^k$ in the analysis, and is hence omitted. 

We are now ready to prove Theorem \ref{thm:mf_oml_regret}.
\begin{proof}[Proof of Theorem \ref{thm:mf_oml_regret}]
Once again, as in the proof of Theorem \ref{FBS_convergence} and Theorem \ref{FBS_convergence_approx}, we prove the regret bounds for generic $\lambda$ and $\eta$ with $\lambda+\eta>0$ and $\lambda\neq \eta$, and then specialize it to $\lambda=0,\,\eta>0$ and $\lambda>0,\,\eta=0$, resp. to derive the claimed conclusions. 

Let $\delta_k:=\delta/k^2$ for $k\geq 1$. By Corollary \ref{cor:epsilon_1_and_epsilon_2}, for any $k\geq 1$ with $n_k=k^3>\frac{2\log(2/\delta_k)}{p_{\min}^2}=\frac{2\log(2/\delta)+4\log k}{p_{\min}^2}$, we have that with probability at most $(1+S)SAT\delta_k$, the inequalities of rewards and transitions estimations in \eqref{eq:epsilon_1} and \eqref{eq:epsilon_2} (with $\delta$ replaced by $\delta_k$) will not all be satisfied. Note that since $\log k\leq k-1<k$ hold for any $k>0$, we have $\log k< k\leq \frac{p_{\min}^2}{8} k^3$ hold for any $k\geq 2\sqrt{2}/p_{\min}$. Hence for any 
\[k\geq k_{\min}^\delta:=\max\left\{2\sqrt{2}/p_{\min},\left(\frac{4\log(2/\delta)}{p_{\min}^2}\right)^{1/3}\right\},
\]
we have 
\[
\dfrac{2\log(2/\delta_k)}{p_{\min}^2}=\dfrac{2\log(2/\delta)+4\log k}{p_{\min}^2}< \dfrac{2\log(2/\delta)}{p_{\min}^2}+k^3/2\leq k^3.
\]

Hence by union bound and the fact that $\sum_{k=1}^{\infty}1/k^2=\pi^2/6$, for any $\delta>0$, we have that \eqref{eq:epsilon_1} and \eqref{eq:epsilon_2} (with $\delta$ replaced by $\delta_k$) hold for all $k\geq k_{\min}^{\delta}$  with probability at least $1-\pi^2(1+S)SAT\delta/6$. For $k<k_{\min}^\delta$, the following naive bounds hold automatically a.s. by the definitions of $\hat{c}^k$ and $\hat{P}^k$:
\[
\|\hat{c}^k-c_{\tilde{R}}(d^k)\|_2\leq 2R_{\max}\sqrt{SAT},\quad \max_{s\in\mathcal{S},a\in\mathcal{A},t=0,\dots,T-1}\sum_{s'\in\mathcal{S}}|\hat{P}_t^k(s'|s,a)-\tilde{P}_t(s'|s,a)|\leq 2.
\]

Hence by Lemma \ref{lemma:deviation_one_policy_nashconv}, Theorem \ref{FBS_convergence_approx} (\cf \eqref{FBS_convergence_approx_general}), Corollary \ref{cor:epsilon_1_and_epsilon_2} and Lemma \ref{Pmod_equivalence}, we have that with probability at least $1-\pi^2(1+S)SAT\delta/6$, for all $k\geq 0$,   
\begin{align*}
\texttt{NashConv}(\boldsymbol{\pi}^{k,l})&\leq \texttt{Expl}(\pi^k)+2C_R\sqrt{\dfrac{\pi}{2}}SAT/\sqrt{N}+6C_RSAT/N\\
&=\texttt{Expl}(\pi^k;\tilde{R},\tilde{P})+2C_R\sqrt{\dfrac{\pi}{2}}SAT/\sqrt{N}+6C_RSAT/N\\
&\leq 2C_R\sqrt{\dfrac{\pi}{2}}SAT/\sqrt{N}+6C_RSAT/N\\
&\quad+ 2T\eta+ \sqrt{SAT}(2{\color{black}T^2}C_R+R_{\max}{\color{black}T})\left(2\left(1-\kappa_{\lambda,\eta}\right)^{\frac{k}{2}}+\sum_{j=0}^{k-1}(1-\kappa_{\lambda,\eta})^{\frac{k-j-1}{2}}\tilde{\epsilon}^j\right), 
\end{align*}
where $\kappa_{\lambda,\eta}:=(\lambda+\eta)^2/(2C_R^2S^2A^2+2\eta^2)$,  $\tilde{\epsilon}^j:=\dfrac{T(T-1)}{2}\epsilon_2^j+T\sqrt{(4+2\alpha\eta+2\alpha R_{\max}SA)(T-1)\epsilon_2^j}+\alpha\epsilon_1^j$, $\alpha=(\lambda+\eta)/(2C_R^2S^2A^2+2\eta^2)$,  $\epsilon_1^k=\eqref{eq:epsilon_1}$ and $\epsilon_2^k=\eqref{eq:epsilon_2}$ (with $\delta$ replaced by $\delta_k$) for $k\geq k_{\min}^\delta$, and $\epsilon_1^k=2R_{\max}\sqrt{SAT}$ and $\epsilon_2^k=2$ otherwise. 

Since $1/(1-\sqrt{1-\kappa_{\lambda,\eta}})=(1+\sqrt{1-\kappa_{\lambda,\eta}})/\kappa_{\lambda,\eta}\leq 2/\kappa_{\lambda,\eta}$, we have by the fact that $n_k=k^3$ is increasing as $k$ grows, 
\[
\sum_{k=0}^{K}n_k\sum_{j=0}^{k-1}(1-\kappa_{\lambda,\eta})^{\frac{k-j-1}{2}}\tilde{\epsilon}^j=\sum_{j=0}^{K-1}\tilde{\epsilon}^j\sum_{k= j+1}^{K}n_k(\sqrt{1-\kappa_{\lambda,\eta}})^{k-j-1}\leq \sum_{k=0}^{K-1}\tilde{\epsilon}^k\dfrac{2n_{K}}{\kappa_{\lambda,\eta}},
\]
and also $\sum_{k=0}^{K}n_k(\sqrt{1-\kappa_{\lambda,\eta}})^k=2n_{K}/\kappa_{\lambda,\eta}$. Hence 
the Nash regret is bounded by 
\begin{align*}
\texttt{NashRegret}(M)&\leq \sum_{k=0}^K\sum_{l=0}^{n_k-1}\texttt{NashConv}(\boldsymbol{\pi}^{k,l})\\
&\leq \left(2C_RSAT\sqrt{\dfrac{\pi}{2N}}+\dfrac{6C_RSAT}{N}+2T\eta\right)\sum_{k=0}^Kn_k\\
&\quad+\sqrt{SAT}(2{\color{black}T^2}C_R+R_{\max}{\color{black}T})\left(2\sum_{k=0}^Kn_k\left(1-\kappa_{\lambda,\eta}\right)^{\frac{k}{2}}+\sum_{k=0}^Kn_k\sum_{j=0}^{k-1}(1-\kappa_{\lambda,\eta})^{\frac{k-j-1}{2}}\tilde{\epsilon}^j\right)\\
&= \left(2C_RSAT\sqrt{\dfrac{\pi}{2N}}+\dfrac{6C_RSAT}{N}+2T\eta\right)\dfrac{K^2(K+1)^2}{4}\\
&\quad + \dfrac{2\sqrt{SAT}K^3(2{\color{black}T^2}C_R+R_{\max}{\color{black}T})}{\kappa_{\lambda,\eta}}\left(2+\sum\nolimits_{k=0}^{K-1}\tilde{\epsilon}^k\right),
\end{align*}
where $K$ is the smallest integer with which $\sum_{k=0}^Kn_k\geq M$. This immediately implies that $\sum_{k=0}^{K-1}n_k=(K-1)^2K^2/4< M$, and hence $(K-1)^4< 4M$. Thus $K<(4M)^{1/4}+1\leq 2(4M)^{1/4}=2\sqrt{2}M^{1/4}$ for any $M\geq 1$ (since $(4M)^{1/4}>1$ for any $M\geq 1$), and similarly for any $M\geq 1$, we have
\[
\sum_{k=0}^Kn_k=\dfrac{K^2(K+1)^2}{4}<\dfrac{((4M)^{1/4}+1)^2((4M)^{1/4}+2)^2}{4}\leq \dfrac{(2(4M)^{1/4})^2\times (3(4M)^{1/4})^2}{4}=36M.
\]

Hence it remains to bound $\sum_{k=0}^{K-1}\tilde{\epsilon}^k$, for which we have 
\begin{align*}
    &\sum_{k=0}^{K-1}\tilde{\epsilon}^k=\sum_{k=0}^{\lceil k_{\min}^\delta\rceil-1}\left(T(T-1)+2T\sqrt{(2+\alpha\eta+\alpha R_{\max}SA)(T-1)}+2\alpha R_{\max}\sqrt{SAT}\right)\\
    & +\alpha\sum_{k=\lceil k_{\min}^\delta\rceil}^{K-1}\left(\sqrt{SAT}\left(C_RSA\left(\frac{1}{N}+\sqrt{\frac{\pi}{2N}}\right)+2R_{\max}\sqrt{\frac{\log(4k^2/\delta)}{p_{\min}n_k}}\right)+C_RS^2AT(T-1)\sqrt{\frac{\log(4k^2/\delta)}{p_{\min}\sum_{j\leq k}n_j}}\right)\\
    & +\sum_{k=\lceil k_{\min}^{\delta}\rceil}^{K-1}\left(T(T-1)S\sqrt{\frac{\log(4k^2/\delta)}{p_{\min}\sum_{j\leq k}n_j}}+2T\sqrt{(2S+\alpha\eta S+\alpha R_{\max}S^2A)(T-1)}\left(\frac{\log(4k^2/\delta)}{p_{\min}\sum_{j\leq k}n_j}\right)^{1/4}\right)\\
    &\leq \left(T^2+2T^{3/2}\sqrt{2+\alpha\eta+\alpha R_{\max}SA}+2\alpha R_{\max}\sqrt{SAT}\right)\left(\max\left\{\frac{2\sqrt{2}}{p_{\min}},\left(\frac{4\log(2/\delta)}{p_{\min}^2}\right)^{1/3}\right\}+1\right)\\
    &+\alpha C_RS^{3/2}A^{3/2}\sqrt{T}K\left(\frac{1}{N}+\sqrt{\frac{\pi}{2N}}\right)+\dfrac{2\alpha R_{\max}\sqrt{SAT}}{\sqrt{p_{\min}}}\sum_{k=1}^{K-1}\frac{\sqrt{\log(4k^2/\delta)}}{k^{3/2}}+\dfrac{2\alpha C_RS^2AT^2}{\sqrt{p_{\min}}}\sum_{k=1}^{K-1}\frac{\sqrt{\log(4k^2/\delta)}}{k^2}\\
    &+\dfrac{2ST^2}{\sqrt{p_{\min}}}\sum_{k=1}^{K-1}\frac{\sqrt{\log(4k^2/\delta)}}{k^2}+\dfrac{2T^{3/2}\sqrt{4S+2\alpha\eta S+2\alpha R_{\max}S^2A}}{p_{\min}^{1/4}}\sum_{k=1}^{K-1}\frac{(\log(4k^2/\delta))^{1/4}}{k}\\
    &=O\left((T^2+\sqrt{SAT})(\log(1/\delta))^{1/3}+S^{3/2}A^{3/2}T^{1/2}\dfrac{M^{1/4}}{\sqrt{N}}\right.\\
&\qquad\qquad\left.+S^2AT^2\sqrt{\log(1/\delta)}+SA^{1/2}T^{3/2}(\log M)^{5/4}(\log(1/\delta))^{1/4}\right)
\end{align*}

Finally, putting all these together, we have that
\begin{align*}
\texttt{NashRegret}(M)&\leq \left(72C_RSAT\sqrt{\dfrac{\pi}{2N}}+\dfrac{216C_RSAT}{N}+72T\eta\right)M\\
&\quad + \dfrac{2\sqrt{SAT}K^3(2{\color{black}T^2}C_R+R_{\max}{\color{black}T})}{\kappa_{\lambda,\eta}}\left(2+\sum\nolimits_{k=0}^{K-1}\tilde{\epsilon}^k\right)\\
&\leq \left(72C_RSAT\sqrt{\dfrac{\pi}{2N}}+\dfrac{216C_RSAT}{N}+72T\eta\right)M +\dfrac{32\sqrt{2SAT}M^{3/4}(2{\color{black}T^2}C_R+R_{\max}{\color{black}T})}{\kappa_{\lambda,\eta}}\\
&\quad + \dfrac{16\sqrt{2SAT}M^{3/4}(2{\color{black}T^2}C_R+R_{\max}{\color{black}T})}{\kappa_{\lambda,\eta}}\sum_{k=0}^{K-1}\tilde{\epsilon}^k\\
&=O\left(\left(\dfrac{S^2A^2{\color{black}T^3}}{\kappa_{\lambda,\eta}\sqrt{N}}+T\eta\right)M+\dfrac{S^2A{\color{black}T^{9/2}}}{\kappa_{\lambda,\eta}}M^{3/4}(\log M)^{5/4}\sqrt{\log(1/\delta)}\right).
\end{align*}
Finally, taking $\eta=\max\{N^{-1/6},M^{-1/12}\}$ when $\lambda=0$ and $\eta=0$ when $\lambda>0$ finishes the proof.
\end{proof}

\section{Additional technical proofs}
In this section, we provide some additional technical proofs of our results. 
\subsection{Proof of Theorem \ref{n2mfg}}\label{n2mfg_proof}
\begin{proof}
It suffices to show that for any policy sequence $\pi\in\Pi$, we have 
\begin{equation}\label{expl_vs_nashconv}
\left|\texttt{Expl}(\pi)-\texttt{NashConv}(\pi)\right|\leq 2C_R\sqrt{\dfrac{\pi}{2}}SAT/\sqrt{N} + 2C_RSAT/N.
\end{equation}
The key is to bound the differences between $L_t^\pi$ and $L_t^{N,\pi}$ for $t\in\mathcal{T}$, where $L_t^{N,\pi}$ denotes the empirical state-action distribution of the $N$-player game under the policy sequence $\pi$. Thanks to the decoupled dynamics $P(s_{t+1}^i|s_t^i,a_t^i)$ of the $N$-player game, the state-action trajectories $\{s_t^i,a_t^i\}_{t\in\mathcal{T}}$ (induced by $\pi$) are independent across $i\in[N]$. Moreover, if we denote $I_{s_0}:=\{i\in[N]|s_0^i=s_0\}$, then for any $s_0\in\mathcal{S}$,  the state-action trajectories $\{s_t^i,a_t^i\}_{t\in\mathcal{T}}$ are indeed i.i.d. across $i\in I_{s_0}$. 
We now show by induction that for any $i\in[N]$ and $t\in\mathcal{T}$,
\begin{equation}\label{expected_empirical=mean_field}
\dfrac{1}{N}\sum_{i\in[N]}\mathbb{P}(s_t^i=s,a_t^i=a)=L_t^{\pi}(s,a),\qquad \forall s\in\mathcal{S},a\in\mathcal{A}.
\end{equation}
For $t=0$, equality \eqref{expected_empirical=mean_field} holds by the fact that $\mathbb{P}(s_0^i=s,a_0^i=a)=\pi_0(a|s)$ if $s=s_0^i$ and $\mathbb{P}(s_0^i=s,a_0^i=a)=0$ otherwise, which immediately leads to 
\[
\dfrac{1}{N}\sum_{i\in[N]}\mathbb{P}(s_0^i=s,a_0^i=a)=\dfrac{1}{N}\sum_{i\in[N]}{\bf 1}\{s_0^i=s\}\pi_0(a|s)=\mu_0^N(s)\pi_0(a|s)=L_0^\pi(s,a). 
\]
Now if equality \eqref{expected_empirical=mean_field} holds for $t$, then 
\begin{align*}
L_{t+1}^{\pi}(s',a')&=\pi_{t+1}(a'|s')\sum_{s\in\mathcal{S},a\in\mathcal{A}}P_t(s'|s,a)\dfrac{1}{N}\sum_{i\in[N]}\mathbb{P}(s_t^i=s,a_t^i=a)\\
&=\dfrac{1}{N}\sum_{i\in[N]}\sum_{s\in\mathcal{S},a\in\mathcal{A}}\mathbb{P}(s_t^i=s,a_t^i=a)P_t(s_{t+1}^i=s'|s_t^i=s,a_t^i=a)\pi_{t+1}(a'|s')\\
&=\dfrac{1}{N}\sum_{i\in[N]}\mathbb{P}(s_{t+1}^i=s')\mathbb{P}(a_{t+1}^i=a'|s_{t+1}^i=s')\\
&=\dfrac{1}{N}\sum_{i\in[N]}\mathbb{P}(s_{t+1}^i=s',a_{t+1}^i=a').
\end{align*}
This completes the induction. Now we are ready to bound the differences between $L_t^{\pi}$ and $L_t^{N,\pi}$ ($t\in\mathcal{T}$). Firstly, notice that by \eqref{expected_empirical=mean_field}, we have 
\begin{align*}
\mathbb{E}\left[L_t^{N,\pi}(s,a)\right]=\dfrac{1}{N}\sum_{i\in[N]}\mathbb{P}(s_t^i=s,a_t^i=a)=L_t^{\pi}(s,a),\qquad\forall s\in\mathcal{S},a\in\mathcal{A}.
\end{align*}
By Hoeffding's inequality, we have that 
\[
\mathbb{P}\left(\left|\dfrac{1}{N}\sum_{i\in [N]}{\bf 1}\{s_t^i=s,a_t^i=a\}-L_t^{\pi}(s,a)\right|\geq \epsilon\right)\leq 2e^{-2N\epsilon^2}.
\]
Hence by the fact that $\mathbb{E}|X|=\int_0^{\infty}\mathbb{P}(|X|\geq\epsilon)d\epsilon$ for any random variable $X$, we have that 
\begin{align*}
&\mathbb{E}\left|L_t^{N,\pi}(s,a)-L_t^{\pi}(s,a)\right|\leq 2\int_0^{\infty}e^{-2N\epsilon^2}d\epsilon=\dfrac{1}{\sqrt{2N}}\int_{-\infty}^{\infty}e^{-x^2}dx=\sqrt{\dfrac{\pi}{2N}}.
\end{align*}
Hence we have
\begin{align*}
\mathbb{E}\left\|L_t^{N,\pi}-L_t^{\pi}\right\|_1&= \sum_{s\in\mathcal{S},a\in\mathcal{A}}\mathbb{E}\left|L_t^{N,\pi}(s,a)-L_t^{\pi}(s,a)\right|
\leq \sqrt{\dfrac{\pi}{2}}SA/\sqrt{N}. 
\end{align*}

Similarly, let $L_t^{N,i,\pi^i,\pi}$ be the empirical state-action distribution of the $N$-player game when agent $i$ takes policy sequence $\pi^i$ while the other agents take policy sequence $\pi$. 
Note that $\mathbb{E}|X-c|=\mathbb{E}|Y-c|$ if $X$ and $Y$ have the same distribution and $c$ is deterministic. Since $L_t^{N,i,\pi^i,\pi}(s,a)$ has the same distribution as $L_t^{N,\pi}(s,a)-\frac{1}{N}\mathbf{1}\{s_t^i=s,a_t^i=a\}+\frac{1}{N}\mathbf{1}\{(s_t^{i})'=s,(a_t^{i})'=a\}$, where $s_t^i$ and $a_t^i$ are  the state and action of the i-th agent taking policy $\pi$ in $L_t^{N,\pi}$, while $(s_t^{i})'$ and $(a_t^{i})'$ are the state and action of the i-th agent taking policy $\pi^i$ in $L_t^{N,i,\pi^i,\pi}$, we have 
\begin{align*}
\mathbb{E}\left\|L_t^{N,i,\pi^i,\pi}-L_t^{\pi}\right\|_1&=\mathbb{E}\sum_{s\in\mathcal{S},a\in\mathcal{A}}\left|L_t^{N,\pi}(s,a)-\frac{1}{N}\mathbf{1}\{s_t^i=s,a_t^i=a\}+\frac{1}{N}\mathbf{1}\{(s_t^{i})'=s,(a_t^{i})'=a\}-L_t^{\pi}(s,a)\right|\\
&\leq \mathbb{E}\|L_t^{N,\pi}-L_t^{\pi}\|_1+2SA/N\leq \sqrt{\dfrac{\pi}{2}}SA/\sqrt{N}+2SA/N.
\end{align*}

Next, we rewrite $\texttt{NashConv}$ and $\texttt{Expl}$ by utilizing the (empirical) mean-field flows. In fact, by the definitions and the occupation measure representation of (single-agent) MDP value functions, we have 
\begin{align*}
\dfrac{1}{N}\sum_{i\in[N]}\max_{\pi^i\in\Pi}V^i(\pi^i,\pi)&=\dfrac{1}{N}\sum_{i\in[N]}\max_{\pi^i\in\Pi}\mathbb{E}_{\pi^i,\pi}\left[\sum_{t\in\mathcal{T}}r_t(s_t^i,a_t^i,L_t^{N,i,\pi^i,\pi})\right],\\
V^{\pi'}(L^{\pi})&=\mathbb{E}_{\pi'}\left[\sum_{t\in\mathcal{T}}r_t(s_t,a_t,L_t^\pi)\right]=\sum_{s\in\mathcal{S},a\in\mathcal{A},t\in\mathcal{T}}d_t^{\pi'}(s,a)R_t(s,a,L_t^{\pi})\\
&=\sum_{s_0\in\mathcal{S}}\mu_0^N(s_0)\sum_{s\in\mathcal{S},a\in\mathcal{A},t\in\mathcal{T}}d_t^{s_0,\pi'}(s,a)R_t(s,a,L_t^{\pi})\\
&=\sum_{s_0\in\mathcal{S}}\dfrac{1}{N}\sum_{i\in[N]}{\bf 1}\{s_0^i=s_0\}\sum_{s\in\mathcal{S},a\in\mathcal{A},t\in\mathcal{T}}d_t^{s_0,\pi'}(s,a)R_t(s,a,L_t^{\pi})\\
&=\dfrac{1}{N}\sum_{i\in[N]}\sum_{s\in\mathcal{S},a\in\mathcal{A},t\in\mathcal{T}}d_t^{s_0^i,\pi'}(s,a)R_t(s,a,L_t^{\pi}),
\end{align*}
where $\mathbb{E}_{\pi'}$ denotes the expectation over the trajectory $\{s_t,a_t\}_{t\in\mathcal{T}}$ resulted from taking policy sequence $\pi'$ in an MDP with  transitions $P_t(s_{t+1}|s_t,a_t)$ and initial state distribution $\mu_0^N$, and $d_t^{\pi'}(s,a):=\mathbb{P}(s_t=s,a_t=a)$ and $d_t^{s_0,\pi'}(s,a):=\mathbb{P}(s_t=s,a_t=a|s_0)$ denote the (initial-state conditioned) state-action occupation-measures of the aforementioned trajectory, with $d_t^{\pi'}(s,a)=\sum_{s_0\in\mathcal{S}}\mu_0^N(s_0)d_t^{s_0,\pi'}(s,a)$. 

Moreover, since $\sum_{s\in\mathcal{S},a\in\mathcal{A},t\in\mathcal{T}}d_t^{s_0^i,\pi'}(s,a)R_t(s,a,L_t^{\pi})$ is the expected value of taking policy $\pi'$ with initial state $s_0^i$ in the previously mentioned $L^\pi$-induced MDP, and from the dynamic programming theory of MDP we know that there exists a policy sequence $\pi^\star$ that achieves the optimal expected value for arbitrary initial states, we have that for this policy sequence $\pi^\star$, \[
\sum_{s\in\mathcal{S},a\in\mathcal{A},t\in\mathcal{T}}d_t^{s_0^i,\pi^\star}(s,a)R_t(s,a,L_t^{\pi})=\max_{\pi'\in\Pi}\sum_{s\in\mathcal{S},a\in\mathcal{A},t\in\mathcal{T}}d_t^{s_0^i,\pi'}(s,a)R_t(s,a,L_t^{\pi}),\quad \forall s_0^i\in\mathcal{S},
\]
which immediately implies that 
\[
\max_{\pi'\in\Pi}\sum_{i\in[N]}\sum_{s\in\mathcal{S},a\in\mathcal{A},t\in\mathcal{T}}d_t^{s_0^i,\pi'}(s,a)R_t(s,a,L_t^{\pi})=\sum_{i\in[N]}\max_{\pi^i\in\Pi}\sum_{s\in\mathcal{S},a\in\mathcal{A},t\in\mathcal{T}}d_t^{s_0^i,\pi^i}(s,a)R_t(s,a,L_t^{\pi})
\]

Hence we have 
\begin{align*}
&\left|\dfrac{1}{N}\sum_{i\in[N]}\max_{\pi^i\in\Pi}V^i(\pi^i,\pi)-\max_{\pi'\in\Pi}V^{\pi'}(L^{\pi})\right|\leq \left|\dfrac{1}{N}\sum_{i\in[N]}\max_{\pi^i\in\Pi}\mathbb{E}_{\pi^i,\pi}\left[\sum_{t\in\mathcal{T}}R_t(s_t^i,a_t^i,L_t^\pi)\right]-\max_{\pi'\in\Pi}V^{\pi'}(L^{\pi})\right|\\
&\qquad+\left|\dfrac{1}{N}\sum_{i\in[N]}\max_{\pi^i\in\Pi}\mathbb{E}_{\pi^i,\pi}\left[\sum_{t\in\mathcal{T}}\left(R_t(s_t^i,a_t^i,L_t^{N,i,\pi^i,\pi})-R_t(s_t^i,a_t^i,L_t^{\pi})\right)\right]\right|\\
&\leq \left|\dfrac{1}{N}\sum_{i\in[N]}\max_{\pi^i\in\Pi}\mathbb{E}_{\pi^i,\pi}\left[\sum_{t\in\mathcal{T}}R_t(s_t^i,a_t^i,L_t^\pi)\right]-\max_{\pi'\in\Pi}\dfrac{1}{N}\sum_{i\in[N]}\sum_{s\in\mathcal{S},a\in\mathcal{A},t\in\mathcal{T}}d_t^{s_0^i,\pi'}(s,a)R_t(s,a,L_t^{\pi})\right|\\
&\qquad + \dfrac{1}{N}\sum_{i\in[N]}\max_{\pi^i\in\Pi}C_R\sum_{t\in\mathcal{T}}\mathbb{E}\|L_t^{N,i,\pi^i,\pi}-L_t^\pi\|_1\\
&=\underbrace{\left|\dfrac{1}{N}\sum_{i\in[N]}\max_{\pi^i\in\Pi}\sum_{s\in\mathcal{S},a\in\mathcal{A},t\in\mathcal{T}}d_t^{s_0^i,\pi^i}(s,a)R_t(s,a,L_t^{\pi})-\dfrac{1}{N}\sum_{i\in[N]}\max_{\pi^i\in\Pi}\sum_{s\in\mathcal{S},a\in\mathcal{A},t\in\mathcal{T}}d_t^{s_0^i,\pi^i}(s,a)R_t(s,a,L_t^{\pi})\right|}_{=0}\\
&\qquad+C_RT\left(\sqrt{\dfrac{\pi}{2}}SA/\sqrt{N} + 2SA/N\right).
\end{align*}

Similarly, we have 
\[
\left|\dfrac{1}{N}\sum_{i\in[N]}V^i(\pi,\pi)-V^{\pi}(L^{\pi})\right|\leq C_R\sqrt{\dfrac{\pi}{2}}SAT/\sqrt{N}. 
\]

Finally, putting together the two mean-field flow approximation errors, we have 
\begin{align*}
\left|\texttt{Expl}(\pi)-\texttt{NashConv}(\pi)\right|&\leq 2C_R\sqrt{\dfrac{\pi}{2}}SAT/\sqrt{N} + 2C_RSAT/N.
\end{align*}
This completes the proof.
\end{proof}

\subsection{Proof of Theorem \ref{thm:mf-omi}}
\begin{proof}
The proof is a direct application of \cite[Lemma 3]{guo2024mf}, which states that if $\pi^\star\in\Pi$ is an NE of the MFG, then  $\exists d^\star\in\mathbb{R}^{SAT}$, such that $\pi^\star\in\texttt{Normalize}(d^\star)$ and that $d^\star$ solves the linear program which minimizes $c(d^\star)^\top d$ subject to $Ad=b$, $d\geq 0$; and conversely, if $d^\star$ solves the linear program which minimizes $c(d^\star)^\top d$ subject to $Ad=b$, $d\geq 0$, then any $\pi^\star\in\texttt{Normalize}(d^\star)$ is an NE of the original MFG. Hence it suffices to show that $d^\star$ solves the linear program which minimizes $c(d^\star)^\top d$ subject to $Ad=b$, $d\geq 0$ if and only if $0\in c(d^\star)+N_{\{x|Ax=b,x\geq 0\}}$. To see this, it suffices to observe that $d^\star$ solves the linear program which minimizes $c(d^\star)^\top d$ subject to $Ad=b$, $d\geq 0$ if and only if $Ad^\star=b,d^\star\geq 0$, and in addition, for any $d$ such that $Ad=b,d\geq 0$, we have $c(d^\star)^\top d\geq c(d^\star)^\top d^\star$. This is exactly $-c(d^\star)\in \mathcal{N}_{\{x|Ax=b,x\geq 0\}}(d^\star)$ by the definition of normal cones. 
\end{proof}

\subsection{Proof of Lemma \ref{consistency_recover}}\label{sec:proof-lemma-consistency_recover}
\begin{proof}
Firstly, by the definition of $A_{\hat{P}}$ and $b$, we have that 
\begin{align}\label{APx=b}
&\sum_{s\in\mathcal{S},a\in\mathcal{A}}x_t(s,a)P_t(s'|s,a)=\sum_{a'\in\mathcal{A}}x_{t+1}(s',a'),\quad \forall s'\in\mathcal{S},t\in\{0,\dots,T-2\},\\
&\sum_{a\in\mathcal{A}}x_0(s,a)=\mu_0(s),\quad \forall s\in\mathcal{S},\quad x_t(s,a)\geq 0,\quad \forall s\in\mathcal{S},a\in\mathcal{A},t\in\mathcal{T}.\notag
\end{align}
Then by the definition of $\Gamma$, for any $\pi\in\Pi$, if we denote $\hat{L}^{\pi}=\Gamma(\pi;\hat{P})$, we have 
\begin{align}\label{Gamma(pi;P)}
\hat{L}_0^\pi(s,a)&=\mu_0(s)\pi_0(a|s),\\
    \hat{L}_{t+1}^{\pi}(s',a')&=\pi_{t+1}(a'|s')\sum_{s\in\mathcal{S},a\in\mathcal{A}}P_t(s'|s,a)\hat{L}_t^\pi(s,a),\quad t\in\{0,\dots,T-2\}.\notag
\end{align}
And in addition, we also have by the definition of \texttt{Normalize} that \begin{align}\label{Normalize_formula}
\pi_t(a|s)\sum_{a'\in\mathcal{A}}x_t(s,a')=x_t(s,a), \quad \forall s\in\mathcal{S},a\in\mathcal{A},t\in\mathcal{T}.
\end{align}
Hence for $t=0$, we have by \eqref{APx=b}, \eqref{Gamma(pi;P)} and \eqref{Normalize_formula} that $\hat{L}_0^{\pi}(s,a)=\pi_0(a|s)\sum_{a'\in\mathcal{A}}x_0(s,a')=x_0(s,a)$ for any $s\in\mathcal{S},a\in\mathcal{A}$. Now suppose that we have $\hat{L}_t^{\pi}(s,a)=x_t(s,a)$ for any $s\in\mathcal{S},a\in\mathcal{A}$. Then for $t+1$, again we have by \eqref{APx=b}, \eqref{Gamma(pi;P)} and \eqref{Normalize_formula} that 
\[
\hat{L}_{t+1}^\pi(s',a')=\pi_{t+1}(a'|s')\sum_{s\in\mathcal{S},a\in\mathcal{A}}P_t(s'|s,a)x_t(s,a)=\pi_{t+1}(a'|s')\sum_{a''\in\mathcal{A}}x_{t+1}(s',a'')=x_{t+1}(s',a').
\]
Hence by induction, we have that $\Gamma(\pi;\hat{P})=\hat{L}^{\pi}=x$. 

Now we prove the second claim. For any $\pi\in\Pi$, if we again denote $\hat{L}^{\pi}=\Gamma(\pi;\hat{P})$, then by \eqref{APx=b} and \eqref{Gamma(pi;P)}, we have that
\begin{align*}
\sum_{a\in\mathcal{A}}\hat{L}_0^\pi(s,a)&=\sum_{a\in\mathcal{A}}\mu_0(s)\pi_0(a|s)=\mu_0(s),\\
    \sum_{a'\in\mathcal{A}}\hat{L}_{t+1}^{\pi}(s',a')&=\sum_{a'\in\mathcal{A}}\pi_{t+1}(a'|s')\sum_{s\in\mathcal{S},a\in\mathcal{A}}P_t(s'|s,a)\hat{L}_t^\pi(s,a)\\
    &=\sum_{s\in\mathcal{S},a\in\mathcal{A}}P_t(s'|s,a)\hat{L}_t^\pi(s,a),\qquad\qquad\qquad t\in\{0,\dots,T-2\},
\end{align*}
namely $A_{\hat{P}}\Gamma(\pi;\hat{P})=b$. Finally, the non-negativity of $\Gamma(\pi;\hat{P})$ is trivial by the non-negativity of $P$, $\mu_0$ and $\pi$ in \eqref{Gamma(pi;P)}. 
\end{proof}

\subsection{Proof of Lemma \ref{MFNE_vs_FP}}\label{sec:proof-lemma-MFNE_VS_FP}
\begin{proof}

Let $d$ be a fixed point of $F_{\alpha,\eta}$. Then by the fact that for a closed convex set $\mathcal{X}$ in some Euclidean space, $\texttt{Proj}_{\mathcal{X}}=(I+\alpha \mathcal{N}_{\mathcal{X}})^{-1}$ for any $\alpha>0$ (\cf \cite[\S 6.1]{ryu2016primer}), we have \[
d-\alpha (c(d)+\eta d)\in (I+\alpha\mathcal{N}_{\{x|Ax=b,x\geq 0\}})(d)=d+\alpha\mathcal{N}_{\{x|Ax=b,x\geq 0\}}(d),
\]
and hence
\begin{equation}\label{fp_to_inclusion}
-(c(d)+\eta d)\in \mathcal{N}_{\{x|Ax=b,x\geq 0\}}(d).
\end{equation}
Hence for any $\beta>0$, we also have 
\[
d-\beta (c(d)+\eta d)\in d+\beta\mathcal{N}_{\{x|Ax=b,x\geq 0\}}(d)=(I+\beta\mathcal{N}_{\{x|Ax=b,x\geq 0\}})(d),
\]
and hence 
\[
\texttt{Proj}_{\{x|Ax=b,x\geq 0\}}(d-\beta(c(d)+\eta d))=d,
\]
namely $d$ is a fixed point of $F_{\beta,\eta}$. This shows that the set of fixed points of $F_{\alpha,\eta}$ is independent of $\alpha>0$ for a given $\eta\geq 0$.

Now suppose that $\pi\in\Pi$ is an NE of the $\eta$-perturbed MFG. Then by Theorem \ref{thm:mf-omi}, we have that 
\[
-\alpha(c(d)+\eta d)\in \alpha \mathcal{N}_{\{x|Ax=b,x\geq 0\}}(d),
\]
and hence 
\[
d-\alpha(c(d)+\eta d)\in (I+\alpha\mathcal{N}_{\{x|Ax=b,x\geq 0\}})(d),
\]
which, given that $\alpha>0$ and hence again $\texttt{Proj}_{\{x|Ax=b,x\geq 0\}}=(I+\alpha \mathcal{N}_{\{x|Ax=b,x\geq 0\}})^{-1}$ as explained above, implies that 
\[
\texttt{Proj}_{\{x|Ax=b,x\geq 0\}}(d-\alpha(c(d)+\eta d))=d,
\]
and hence $d$ is a fixed point of $F_{\alpha,\eta}$. 

Finally, suppose that $F_{\alpha,\eta}(d)=d$. Then by \eqref{fp_to_inclusion} and Theorem \ref{thm:mf-omi}, we immediately conclude that any $\pi\in\texttt{Normalize}(d)$ is an NE of the MFG with the $\eta$-perturbed rewards $\hat{r}_t^{\eta}$. In addition, by Lemma \ref{consistency_recover}, we also have $d=L^\pi=\Gamma(\pi;P)$ and $Ad=b,d\geq 0$.
\end{proof}

\subsection{Proof of Corollary \ref{FBS_iteration_complexity}}\label{sec:proof-coro-FBS_iteration_complexity}
\begin{proof} We prove the results for $\lambda=0$ and $\lambda>0$ separately below.

\medskip

\noindent\textit{Case 1: $\lambda=0$.}  The choices of $\alpha$ and $\eta$ are simply replacing $\epsilon$ with $\epsilon/(4T)$ in Theorem \ref{FBS_convergence} in the case when $\lambda=0$, and hence we have 
\[
\texttt{Expl}(\pi^k)\leq \epsilon/2+ 2\sqrt{SAT}(2{\color{black}T^2}C_R+R_{\max}{\color{black}T})\left(1-\kappa_{\frac{\epsilon}{4T}}\right)^{\frac{k}{2}},
\]
where $\kappa_{\frac{\epsilon}{4T}}=\epsilon^2/\left(32C_R^2S^2A^2T^2+2\epsilon^2\right)$. Hence to achieve $\texttt{Expl}(\pi^k)\leq \epsilon$, it suffices to have 
\[
k\geq 2\log\dfrac{4\sqrt{SAT}(2{\color{black}T^2}C_R+R_{\max}{\color{black}T})}{\epsilon}/\log(1/(1-\kappa_{\frac{\epsilon}{4T}})).
\]
Since $\log(1-\kappa_{\frac{\epsilon}{4T}})\leq -\kappa_{\frac{\epsilon}{4T}}$ (as $\kappa_{\frac{\epsilon}{4T}}\in(0,1)$), it suffices to have 
\[
k\geq \dfrac{64C_R^2S^2A^2T^2+4\epsilon^2}{\epsilon^2}\times\log\dfrac{4\sqrt{SAT}(2{\color{black}T^2}C_R+R_{\max}{\color{black}T})}{\epsilon},
\]
namely $k=\Omega\left(\epsilon^{-2}\log(1/\epsilon)\right)$. 

\medskip

\noindent\textit{Case 2: $\lambda>0$.} By Theorem \ref{FBS_convergence}, we have that $\texttt{Expl}(\pi^k)\leq \epsilon$ if 
\[
2\sqrt{SAT}(2{\color{black}T^2}C_R+R_{\max}{\color{black}T})(1-\kappa)^{k/2}\leq \epsilon,
\]
which is equivalent to $k\geq \log\dfrac{2\sqrt{SAT}(2{\color{black}T^2}C_R+R_{\max}{\color{black}T})}{\epsilon}/\log(1/(1-\kappa))$. Hence to have $\texttt{Expl}(\pi^k)\leq \epsilon$, it suffices to have 
\[k\geq \log\dfrac{2\sqrt{SAT}(2{\color{black}T^2}C_R+R_{\max}{\color{black}T})}{\epsilon}/\kappa=\dfrac{2C_R^2S^2A^2}{\lambda^2}\log\dfrac{2\sqrt{SAT}(2{\color{black}T^2}C_R+R_{\max}{\color{black}T})}{\epsilon},
\]
namely $k=\Omega(\log(1/\epsilon))$.
\end{proof}

\subsection{Proof of Lemmas \ref{projection_d_vs_L} and \ref{lipschitz_gamma_pi_P}}\label{sec:proof-lemma-projection_d_vs_L}
\begin{proof}[Proof of Lemma \ref{projection_d_vs_L}]
By Lemma \ref{consistency_recover}, let $\bar{\pi}\in\texttt{Normalize}(\hat{d})$, then $\Gamma(\bar{\pi};\hat{P})=\hat{d}$. In addition, by the same lemma, we also have $A_{\hat{P}}\Gamma(\hat{\pi};\hat{P})=b,\Gamma(\hat{\pi};\hat{P})\geq 0$. Hence by the optimality $\hat{\pi}$, we have that  \[
\|\hat{d}-\tilde{d}\|_2=\|\Gamma(\bar{\pi};\hat{P})-\tilde{d}\|_2\geq \|\Gamma(\hat{\pi};\hat{P})-\tilde{d}\|_2,
\] and also by the optimality of $\hat{d}$ and the feasibility of $\Gamma(\hat{\pi};\hat{P})$ for the constraint $A_{\hat{P}}d=b,d\geq 0$, we have 
\[
\|\Gamma(\hat{\pi};\hat{P})-\tilde{d}\|_2\geq \|\hat{d}-\tilde{d}\|_2.
\]
Hence we have $\|\Gamma(\hat{\pi};\hat{P})-\tilde{d}\|_2= \|\hat{d}-\tilde{d}\|_2$. By the uniqueness of the solution to strongly convex optimization problem of minimizing $\|d-\tilde{d}\|_2^2$ over the set of $d$ with $A_{\hat{P}}d=b,d\geq 0$ (or by the uniqueness of projection onto a closed convex set), we conclude that $\hat{d}=\Gamma(\hat{\pi};\hat{P})$.
\end{proof}

\begin{proof}[Proof of Lemma \ref{lipschitz_gamma_pi_P}]
Let $\hat{L}^\pi=\Gamma(\pi;\hat{P})$ and $L^\pi=\Gamma(\pi;\tilde{P})$. Then by the definition of $\Gamma$, we have 
    \begin{align*}
    \left|\hat{L}_{t+1}^\pi(s',a')-L_{t+1}^\pi(s',a')\right|=\pi_{t+1}(a'|s')\left|\sum_{s\in\mathcal{S},a\in\mathcal{A}}\left(\hat{P}_t(s'|s,a)\hat{L}_t^{\pi}(s,a)-\tilde{P}_t(s'|s,a)L_t^{\pi}(s,a)\right)\right|,
    \end{align*}
and hence 
\begin{align*}
    \|\hat{L}_{t+1}^{\pi}-L_{t+1}^{\pi}\|_1&\leq \sum_{s'\in\mathcal{S},s\in\mathcal{S},a\in\mathcal{A}}\left(\left|\hat{P}_t(s'|s,a)-\tilde{P}_t(s'|s,a)\right|\hat{L}_t^{\pi}(s,a)+\left|\hat{L}_t^{\pi}(s,a)-L_t^{\pi}(s,a)\right|\tilde{P}_t(s'|s,a)\right)\\
    &\leq \max_{s\in\mathcal{S},a\in\mathcal{A}}\sum_{s'\in\mathcal{S}}|\hat{P}_t(s'|s,a)-\tilde{P}_t(s'|s,a)|+\|\hat{L}_t^{\pi}-L_t^{\pi}\|_1,
\end{align*}
which, together with the fact that $\hat{L}_0^{\pi}(s,a)=\mu_0(s)\pi_0(a|s)=L_0^{\pi}(s,a)$, implies that 
\[
    \|\hat{L}_t^{\pi}-L_t^{\pi}\|_1\leq t \max_{s\in\mathcal{S},a\in\mathcal{A},t=0,\dots,T-2}\sum_{s'\in\mathcal{S}}|\hat{P}_t(s'|s,a)-\tilde{P}_t(s'|s,a)|,
\]
and hence 
\[
\|\hat{L}^{\pi}-L^{\pi}\|_1\leq \dfrac{T(T-1)}{2}\max_{s\in\mathcal{S},a\in\mathcal{A},t=0,\dots,T-2}\sum_{s'\in\mathcal{S}}|\hat{P}_t(s'|s,a)-\tilde{P}_t(s'|s,a)|.
\]
This completes the proof.
\end{proof}




 






{\color{black}
\section{Implementation and experiments}
In this section, we highlight some notable implementation details for both \texttt{MF-OMI-FBS} (Algorithm \ref{MF-OMI-FBS-Cons}) and \texttt{MF-OML} (Algorithm \ref{MF-OML}). We then evaluate the performance of \texttt{MF-OMI-FBS} (which we also sometimes refer to as \texttt{OccupationMeasureInclusion} in the numerical experiments below) and compare it with the baseline algorithms in the MFG literature, including Online Mirror Descent (OMD) \cite{perolat21}, Fictitious Play (FP) \cite{perrin2020fictitious}, Prior Descent (PD) \cite{cui2021approximately}, and MF-OMO \cite{guo2022mf}. Our implementations 
are based on \texttt{MFGLib} \cite{guo2023mfglib}. 

In all the experiments, we always set $\eta=0$ for \texttt{MF-OMI-FBS} as it turns out to consistently outperform $\eta>0$ choices in practice. For each algorithm, we tune their hyper-parameters via both optuna \cite{akiba2019optuna} (via its support in MFGLib) and manual grid search, and show the best performing choice for each algorithm in the final plots for clarity. It turned out that \texttt{MF-OMI-FBS} works perfectly with only optuna tuning, while all other baselines generally requires further manual gridding.   







\subsection{Implementation tips}\label{impl_tips}
In this section, we explain some implementation details of \texttt{MF-OMI-FBS} and \texttt{MF-OML}, which demonstrate how the projection step in each iteration can be made efficient and how the algorithm is extended to more general settings where the transition probabilities are mean field dependent.

\paragraph{Rewriting consistency projection for OSQP interface.} To utilize OSQP \cite{stellato2020osqp}, we rewrite the projection step into the following stand-form:
\begin{equation*}
\begin{array}{ll}
\text{minimize}_d & \dfrac{1}{2}d^\top Pd+q^\top d\\
\text{subject to} & l\leq Ad \leq u,
\end{array}
\end{equation*}
where $P:=2I$, $q=-2\tilde{d}^{k+1}$, $l=\left[\begin{matrix}b\\0\end{matrix}\right]$, $u=\left[\begin{matrix}b\\\infty\end{matrix}\right]$. 

In the implementation, we set the OSQP solution precision to \texttt{1e-8}. This is extremely fast in general and is sufficient to guarantee the fast convergence of \texttt{MF-OMI-FBS}. However, it would be interesting to study the trade-off between per-iteration costs involved in the projection steps and the accuracies as we vary the precision of OSQP.

{\color{black}\paragraph{Warmstart for OSQP.} As iterations proceed and solutions converge, it is natural to warm-start the inner optimization of the quadratic programs (QP) for consistency projections, as the projected iterates of consecutive outer iterations would be closer and 
closer. We hence enable warm-start of the primal and dual variables from previous QP iterations in OSQP. We have found that such a simple implementation level optimization consistently improves the per-iteration runtime, and in the best case (\eg, in the building evacuation example below) achieves up to four times acceleration.}

\paragraph{Mean-field dependent dynamics.} To handle generic mean-field dependent dynamics, we replace $A$ with $A_{\tilde{d}^{k+1}}$. With this simple change, the projection step remains a convex quadratic program and hence can be solved by passing the aforementioned standard form (with $A$ replaced by $A_{\tilde{d}^{k+1}}$) to OSQP. As we will see in the empirical results below, this trick turns out to work extremely well in practice. Figuring out the underlying mechanism and establishing theoretical guarantees for this heuristic implementation for generic mean-field dependent dynamics would be a very interesting future work. 

 
\paragraph{Speed-up compared to \texttt{MF-OMO}.} The per-iteration cost of \texttt{MF-OMI-FBS} significantly improves over \texttt{MF-OMO} as the latter requires gradient computations while the former does not. The only slight speed uncertainty comes from the OSQP subproblem solving time, and we leave it as an interesting future problem to study the best scheduling of the OSQP target accuracies, etc. as iterations proceed. 

\paragraph{Tips for online RL.} The cubic growth of $n_k=k^3$ in Theorem \ref{thm:mf_oml_regret} turns out to be too aggressive and we found that in general even constant choices of $n_k$ work very well in practice. Particularly, in our online RL experiments, we choose $n_k=20$. The choice of $p_0$ in the default modification is simply chosen as the uniform distribution over states. For simplicity, we use only samples from outer iteration $k$ in \texttt{MF-OML} for estimating $\hat{P}^k$, which turns out to suffice in practice. On the other hand, it would be interesting to compare the effect of data reuse and to explore function approximation to allow for more efficient reusing of samples for both transition and dynamics. We leave this for future work.   



\subsection{Empirical evaluation of \texttt{MF-OMI-FBS}}
In this section, we show the performance of \texttt{MF-OMI-FBS} against OMD, FP, PD and MF-OMO on three different problems, including both MFGs with monotone rewards and mean-field-independent dynamics, as studied in the paper, as well as more general MFGs that lack monotone rewards and have transition probabilities that depend on the mean field. The implementation of all environments and baseline algorithms can be found in \url{https://github.com/radar-research-lab/MFGLib}. Additionally, problem dimensions can be found on the top title of each plot (and the meanings of each dimension parameter can be found in \texttt{MFGLib}). 




\paragraph{Building evacuation.}
This problem involves a multilevel building where a crowd of agents aims to evacuate by descending to the ground floor as quickly as possible while maintaining social distancing. Each floor features two staircases located at opposite corners, requiring agents to traverse the entire floor to reach the next staircase. Agents can move in four cardinal directions (up, down, left, right), stay in place, or transition between floors when positioned on a staircase. Each agent in the crowd aims to descend to the bottom floor as quickly as possible, while minimizing crowding. 



\begin{figure}[h]
\centering
\includegraphics[width=0.45\textwidth]{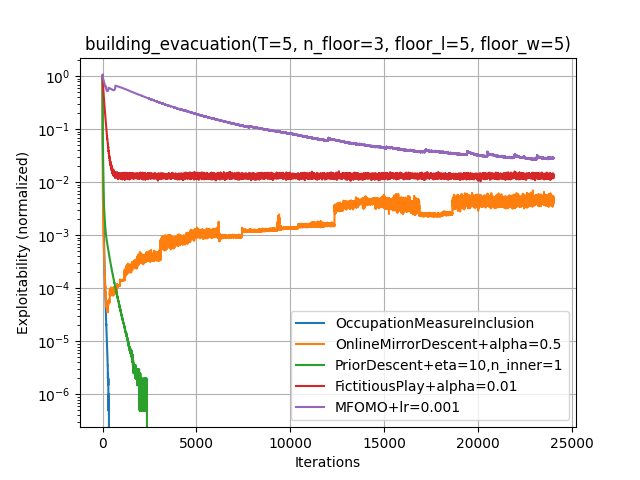}
\includegraphics[width=0.45\textwidth]{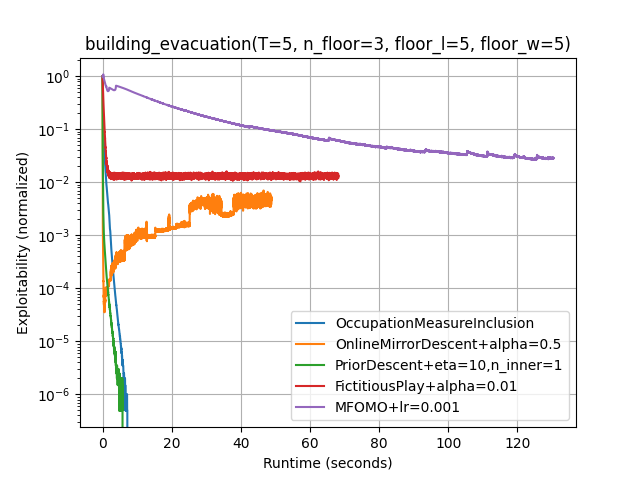}
\caption{Building evacuation. Left: comparison on number of iterations. Right: comparison on runtime.}
\label{be_omi}
\end{figure}

\paragraph{Random linear.} This environment is designed with rewards and transition probabilities defined as random affine functions of the mean-field. To ensure the validity of transition probabilities, a softmax function is applied to the output of the affine functions. 



\begin{figure}[h]
\centering
\includegraphics[width=0.45\textwidth]{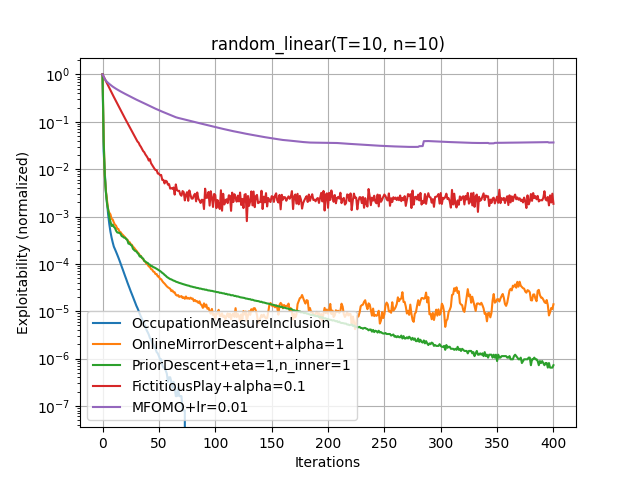}
\includegraphics[width=0.45\textwidth]{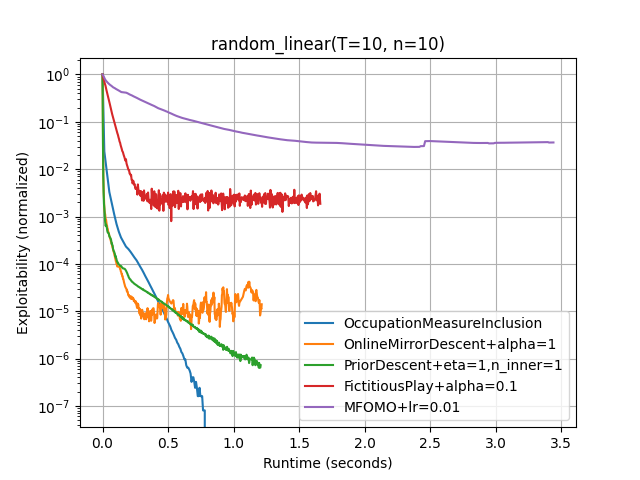}
\caption{Random linear. Left: comparison on number of iterations. Right: comparison on runtime.}
\label{rl_omi}
\end{figure}



\paragraph{SIS.} This problem studies a simple pandemic model. At each time step, agents choose between social distancing or going out. Susceptible agents who go out risk infection with a probability proportional to the number of infected agents, while those who social distance remain healthy. Infected agents recover with a fixed probability per time step. Agents aim to minimize their costs associated with social distancing and infection. The parameters are chosen the same as in \cite{cui2021approximately}.


\begin{figure}[h]
\centering
\includegraphics[width=0.45\textwidth]{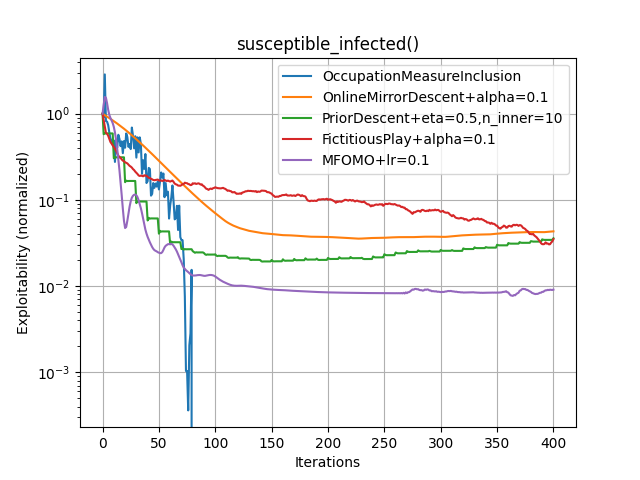}
\includegraphics[width=0.45\textwidth]{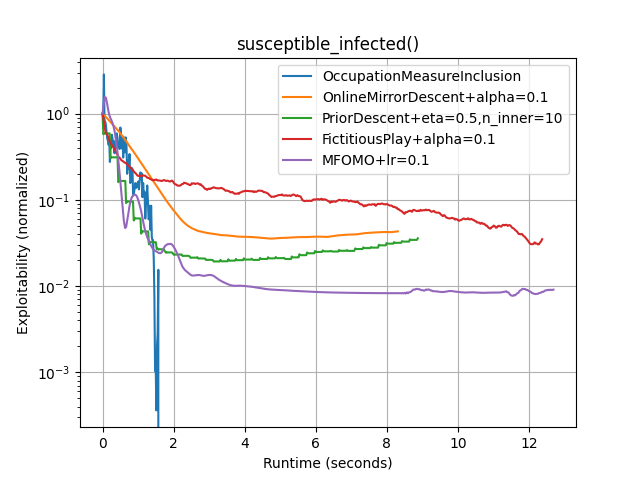}
\caption{SIS. Left: comparison on number of iterations. Right: comparison on runtime.}
\label{sis_omi}
\end{figure}

\paragraph{Observations.} 
Among all the three problems, building evacuation involves a monotone MFG with mean-field independent dynamics that satisfies the assumption of our paper. The MFGs of the other two problems, however, do not satisfy the monotonicity assumption and have mean-field dependent transition probabilities. 
{\color{black} 

From Figures \ref{be_omi}, \ref{rl_omi}, and \ref{sis_omi}, we observe that \texttt{MF-OMI-FBS} rapidly achieves an exploitability of $10^{-6}$ or lower across all problems, consistently outperforming all baseline algorithms in both convergence rate and runtime. The only minor exception is the runtime for the building evacuation problem, where Prior Descent (PD) performs slightly faster. However, this appears to be primarily due to the absence of further implementation-level optimizations in our current setup. In particular, as noted in Section \ref{impl_tips}, enabling warm-starting in OSQP has already reduced the per-iteration runtime of \texttt{MF-OMI-FBS} by a factor of four. We believe that with more advanced quadratic programming solvers and additional implementation improvements, the efficiency of the inner projection step can be further enhanced, potentially eliminating the remaining performance gap.}

\subsection{Online RL}
In this section, we evaluate the performance of \texttt{MF-OML} in the online RL setup. Here we consider the SIS environment again, with $T=4$ for simplicity. We set $\alpha=0.02$ and consider $N=3,6,20$ number of players and run \texttt{MF-OML} for 50 outer iterations (so a total of $1000$ episodes given that we take $n_k=20$). To account for the randomness, we run 10 simulations for each instance of $N$ and plot the 95\% confidence intervals. 

Note that since evaluating \texttt{NashConv} for $N$-player games involves calculating the distribution of the empirical distribution of the $N$ players, which has exponential complexity, it is hence challenging to numerically demonstrate the performance of the MF-OML algorithm in the $N$-player RL setting. 
As a surrogate, we show the performance of MF-OML in terms of exploitability for MFGs instead of  \texttt{NashConv} for $N$-player games. The regret of exploitability  captures the performance of the algorithm up to the inherent and algorithm independent mean-field approximation error. More precisely, if we define 
$\texttt{ExplRegret}(M):=\sum_{m=0}^{M-1}\texttt{Expl}(\pi^m)$, then we have 
\begin{equation*}
\begin{split}
\text{bound of \texttt{NashRegret}($M$)}&= \text{\texttt{ExplRegret}(M)} + \left(72C_RSAT\sqrt{\frac{\pi}{2N}}+\frac{96C_RSAT}{N}\right)M \\
&= \text{\texttt{ExplRegret}(M)} + O(M/\sqrt{N}).
\end{split}
\end{equation*}
To the best of our knowledge, accurately evaluating \texttt{NashConv} for $N$-player games is still an open problem in the multi-agent RL literature except for very small problems \cite{muller2019generalized}, and is hence left for future work. 

From Figure \ref{sis_online_rl}, we can see that the regret indeed grows sub-linearly (and the mean-field approximation error internal to \texttt{ExplRegret} which comes from Proposition \ref{prop:RP_est_error} that grows linearly is dominated). In addition, we can see that as $N$ grows, the regret gradually decreases. These validate our theoretical claims of the regret bounds of \texttt{MF-OML}.




\begin{figure}[h]
\centering
\includegraphics[width=0.75\textwidth]{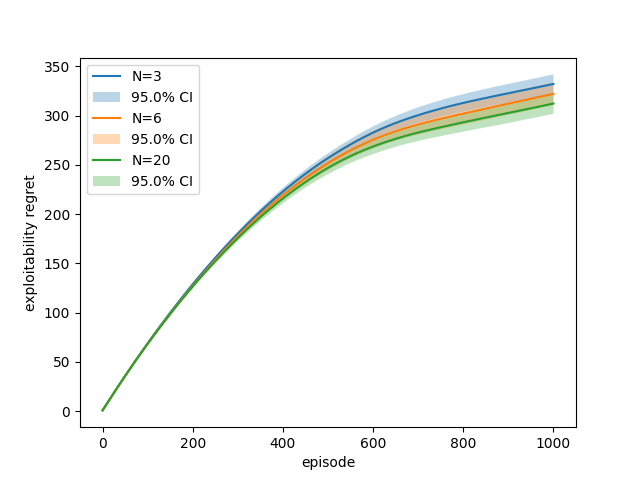}
\caption{SIS exploitability regret}
\label{sis_online_rl}
\end{figure}


}

\clearpage
\newpage
\bibliography{mfoml}
\bibliographystyle{plain}
\end{document}